\documentclass[a4paper,leqno]{amsart}

\usepackage[english]{babel} 									
\usepackage[T1]{fontenc}							
\usepackage[utf8]{inputenx}
				
\usepackage{mathtools}								
\usepackage{empheq}									
\usepackage{enumitem}								

\usepackage{amssymb}								
\usepackage[sc]{mathpazo}							
\usepackage{mathrsfs}								

\usepackage{verbatim}                               
\usepackage{iflang}									
\usepackage{xifthen}								

\usepackage[final]{pdfpages}						
\usepackage{todonotes}								
\usepackage{aliascnt}								
\usepackage{needspace}								
\usepackage{refcount}

\usepackage{subcaption}			
\usepackage{placeins} 			
\usepackage{grffile}			
\usepackage[all]{xy}			

\usepackage{transparent}
\usepackage{color}				

\usepackage{graphicx}			

\usepackage[babel]{csquotes}							

\usepackage{nameref}									
\usepackage[colorlinks	=	true,						
          	 	linkcolor	=	red,
            	urlcolor	=	red,
            	citecolor	=	red]%
            	{hyperref}
\usepackage{bookmark}									
\usepackage{cleveref}									
\numberwithin{equation}{section}

\newtheorem{theorem}{Theorem}[section]

\newtheorem{corollary}[theorem]{Corollary}
\newtheorem{lemma}[theorem]{Lemma}
\newtheorem{proposition}[theorem]{Proposition}

\newtheorem{definition}[theorem]{Definition}

\theoremstyle{remark}
\newtheorem{remark}[theorem]{Remark}

\theoremstyle{remark}

\newcommand{\N}{\mathbb{N}}
\renewcommand{\S}{\mathbb{S}}

\newcommand{\R}{\mathbb{R}}

\newcommand{\lm}{\mathscr{L}}

\DeclareMathOperator{\diam}{diam}

\DeclareMathOperator{\LIP}{LIP}

\DeclareMathOperator{\trace}{tr}

\allowdisplaybreaks

\newcommand{\apmd}[2][]{							
	\ifthenelse{\equal{#1}{}}%
					{ \operatorname{N}_{#2}	}%
					{ \operatorname{N}_{#1,#2} 	}}

\newcommand{\aint}[2][]{
	\ifthenelse{\equal{#1}{}}%
					{%
\mathchoice%
      {\mathop{\kern 0.2em\vrule width 0.6em height 0.69678ex depth -0.58065ex
              \kern -0.8em \intop}\nolimits_{\kern -0.45em#2}^{#1}}%
      {\mathop{\kern 0.1em\vrule width 0.5em height 0.69678ex depth -0.60387ex
              \kern -0.6em \intop}\nolimits_{#2}^{#1}}%
      {\mathop{\kern 0.1em\vrule width 0.5em height 0.69678ex depth -0.60387ex
              \kern -0.6em \intop}\nolimits_{#2}^{#1}}%
      {\mathop{\kern 0.1em\vrule width 0.5em height 0.69678ex depth -0.60387ex
              \kern -0.6em \intop}\nolimits_{#2}^{#1}}}%
					{%
\mathchoice%
      {\mathop{\kern 0.2em\vrule width 0.6em height 0.69678ex depth -0.58065ex
              \kern -0.8em \intop}\nolimits_{\kern -0.45em#1}^{#2}}%
      {\mathop{\kern 0.1em\vrule width 0.5em height 0.69678ex depth -0.60387ex
              \kern -0.6em \intop}\nolimits_{#1}^{#2}}%
      {\mathop{\kern 0.1em\vrule width 0.5em height 0.69678ex depth -0.60387ex
              \kern -0.6em \intop}\nolimits_{#1}^{#2}}%
      {\mathop{\kern 0.1em\vrule width 0.5em height 0.69678ex depth -0.60387ex
              \kern -0.6em \intop}\nolimits_{#1}^{#2}}}}

\usepackage{thmtools}
\usepackage{thm-restate}

\title{Ultralimits of Sobolev maps and stability of Dehn functions}
\author{Toni Ikonen}
\address{Department of Mathematics, University of Fribourg, Chemin du Musée 23, 1700 Fribourg, Switzerland.}
\email{toni.ikonen@unifr.ch}

\author{Stefan Wenger}
\address{Department of Mathematics, University of Fribourg, Chemin du Musée 23, 1700 Fribourg, Switzerland.}
\email{stefan.wenger@unifr.ch}

\thanks{Both authors were supported by Swiss National Science Foundation grant 212867. T.I. was also supported by the Research Council of Finland, project number 332671.}

\begin{document}

\begin{abstract}
    We show that the ultralimit of a bounded sequence of Lipschitz maps into pointed metric spaces extends naturally to $p$-bounded sequences of Sobolev maps and that this ultralimit for Sobolev maps enjoys desirable properties. We use this to prove the stability of Dehn functions under ultraconvergence of pointed length spaces, thus resolving a problem posed by several researchers in the field. As an application, we obtain a simpler proof of a recent result of Stadler--Wenger, previously proved in the locally compact case by Lytchak--Wenger, characterizing spaces of curvature bounded above by $\kappa$ via an isoperimetric inequality for curves.
\end{abstract}
\maketitle

\section{Introduction and main results}

\subsection{Background}\label{section-background}
Limits of metric spaces and mappings are important in many areas of geometry and analysis. In the case of spaces, the (pointed) Gromov--Hausdorff convergence, introduced by Gromov in \cite{gromov-1981-groups-of-polynomial-growth-and-expanding-maps}, has wide applicability for example in geometric group theory, in the theory of Ricci limit spaces, in the theory of synthetic curvature lower bounds, in Riemannian/subRiemannian and metric geometry and in geometric analysis \cite{pansu-1983-croissance-des-boules-et-geodesiques-fermees-dans-les-nilvarietes,petersen-1990-a-finiteness-theorem-for-metric-spaces,grove-petersen-wu-1990-geometric-finiteness-theorems-via-controlled-topology,cheeger-colding-1997-on-the-structure-of-spaces-with-Ricci-curvature-bounded-below-I,cheeger-colding-1997-on-the-structure-of-spaces-with-Ricci-curvature-bounded-below-II,cheeger-colding-1997-on-the-structure-of-spaces-with-Ricci-curvature-bounded-below-III,sturm-2006-on-the-geometry-of-metric-measure-spaces-I,sturm-2006-on-the-geometry-of-metric-measure-spaces-II,Gro:07,wenger-2008-gromov-hyperbolic-spaces-and-the-sharp-isoperimetric-constant,lott-villani-2009-ricci-curvature-for-metric-measure-spaces-via-optimal-transport,brue-naber-semola-2022-boundary-regularity-and-stability-for-spaces-with-ricci-bounded-below,bate-2022-characterising-rectifiable-metric-spaces-using-tangent-spaces,brue-pigati-semola-2025-topological-regularity-and-stability-of-noncollapsed-spaces-with-ricci-curvature-bounded-below}. However, the required assumptions often do not hold in the generality necessary for applications in geometric group theory, for spaces with curvature bounded above in the sense of Alexandrov, and for diameter-volume bounded sequences of closed Riemannian manifolds. A more flexible notion of convergence is given by the ultralimit $X_\omega$ of a sequence of pointed metric spaces $X_m= (X_m, d_m, p_m)$ with respect to some non-principal ultrafilter $\omega$ on $\N$. This notion, first introduced by van den Dries and Wilkie \cite{van-der-dries-wilkie-1984-gromovs-theorem-on-groups-of-polynomial-growth-and-elementary-logic},  has played a key role in geometric group theory and the theory of curvature upper bounds, in particular, in the study of asymptotic cones of groups and spaces \cite{gromov-1993-asymptotic-invariants-of-infinite-groups,kapovich-leeb-1995-on-asymptotic-cones-and-quasi-isometry-classes-of-fundamental-groups-of-3-manifolds,kleiner-leeb-1997-rigidity-of-quasi-isometries-for-symmetric-spaces-and-euclidean-buildings,kapovich-kleiner-1998-quasi-isometries-and-the-de-rham-decomposition,drutu-2002-quasi-isometry-invariants-and-asymptotic-cones,kramer-shelah-tent-thomas-2005-asymptotic-cones-of-finitely-presented-groups,drutu-sapir-2005-tree-graded-spaces-and-asymptotic-cones-of-groups}. We refer to \Cref{section-ultralimits-of-spaces} for the definitions of ultralimits.

Similar flexibility is desirable in the context of mappings $\varphi_m \colon ( X_m, d_m, p_m ) \to ( Y_m, d_m, q_m )$. This is achieved as follows in the Lipschitz setting. We say that $( \varphi_m )$ is \emph{Lipschitz bounded} if
\begin{equation}\label{equation-Lipschitz-boundedness}
    \sup_{ m \in \mathbb{N} } \left( d_m( q_m, \varphi_m(p_m) ) + \LIP( \varphi_m ) \right) < \infty,
\end{equation}
where $\LIP( \varphi )$ is the Lipschitz constant. Any such sequence has a \emph{pointwise ultralimit}
\begin{align*}
    \varphi_\omega \colon X_\omega \to Y_\omega,
    \quad\text{where $\varphi_\omega( \lim\nolimits_\omega x_m ) = \lim\nolimits_\omega \varphi_m(x_m)$,}
\end{align*}
which is again Lipschitz; see \Cref{section-ultralimits-of-mappings}. This flexibility is already useful in the case of a constant sequence $X_m = X$ in which case we obtain a pointwise ultralimit $X \to Y_\omega$ by precomposing $\varphi_\omega$ with the canonical isometric embedding of $X$ into the \emph{ultracompletion} $X_\omega$.

In many problems in calculus of variations and geometric group theory, Lip\-schitz maps form too rigid of a class. Therefore it is desirable to extend the ultralimit construction to more general classes of mappings such as Sobolev maps. In fact, there exists a robust metric valued Sobolev theory on Euclidean domains or more general metric spaces; see e.g. \cite{Kor:Sch:93,Haj:96,Res:97,Hei:Kos:Sha:Ty:01,Hei:Kos:Sha:Ty:15}. These have become prominent in various fields including analysis in metric spaces, metric geometry, and geometric group theory \cite{gromov-schoen-1992-harmonic-maps-into-singular-spaces-and-p-adic-superrigidity-for-lattices-in-groups-of-rank-one,jost-1994-equilibrium-maps-between-metric-spaces,Schoen:Wolfson:01,daskalopoulos-mese-2011-superrigidity-of-hyperbolic-buildings,Lyt:Wen:17:areamini,Lyt:Wen:17:energyarea,Raj:17,Lyt:Wen:18:intrinsic,Lyt:Wen:18:CAT,Lyt:Wen:20,daskalopoulos-mese-2021-rigidity-of-teichmuller-space,Nta:Rom:22,ntalampekos-romney-2023-polyhedral-approximation-of-metric-surfaces-and-applications-to-uniformization,meier-wenger-2025-quasiconformal-almost-parametrizations-of-metric-surfaces}.

We focus on  $W^{1,p}$-Sobolev maps $( u_m \colon \Omega \to X_m )$, where $\Omega \subset \mathbb{R}^n$ is a bounded Lipschitz domain, $p \in (1,\infty)$, and $(X_m)$ is a sequence of \emph{complete} pointed metric spaces. We say that a sequence of $W^{1,p}$-Sobolev maps $( u_m \colon \Omega \to X_m )$ is \emph{$p$-bounded} if 
\begin{equation}\label{eq:bounded-Sobolev-sequence}
    \sup_{m\in\N} \left(\int_\Omega d^p_m(p_m, u_m(z))\,dz + E_+^p(u_m)\right)<\infty,
\end{equation}
where $E_{+}^p(u)$ is the energy of $u$; we refer to \Cref{sec:sobolevmaps} for definitions and notation related to Sobolev maps.

In the present manuscript, we construct an ultralimit for every $p$-bounded sequence. We show that it has many natural properties and use it to establish a stability property for the Dehn function under ultralimits. Roughly speaking, Dehn functions measure how difficult it is to fill in closed Lipschitz curves of a given length by discs. They are closely related to the isoperimetric inequality of the ambient space and they have been widely studied in many different settings; see e.g. \cite{reshetnyak-1968-non-expansive-maps-in-a-space-of-curvature-no-greater-than-K,gromov-1993-asymptotic-invariants-of-infinite-groups,gersten-holt-riley-2003-isoperimetric-inequalities-for-nilpotent-groups,young-2013-the-dehn-function-of-slnz,Lyt:Wen:18:CAT,stadler-the-structure-of-minimal-surfaces-in-cat-0-spaces,stadler-wenger-2025-isoperimetric-inequalities-vs-upper-curvature-bounds}. 
The question on the stability of Dehn functions has only been addressed in special cases in \cite{wenger-2011-nilpotent-groups-without-exactly-polynomial-dehn-function,Wen:2019,Lyt:Wen:You:20,stadler-wenger-2025-isoperimetric-inequalities-vs-upper-curvature-bounds}. We also provide further applications.

\subsection{Ultralimits of Sobolev maps}
Our first theorem shows that we can extend the ultralimit construction to the Sobolev setting. For the statement, let $\Omega\subset\R^n$ be a bounded Lipschitz domain and $p \in (1,\infty)$. Let $\omega$ be a non-principal ultrafilter on $\N$.

\begin{theorem}\label{theorem-ultralimit-construction}
There is a unique map sending a $p$-bounded sequence $( u_m \colon \Omega \to X_m )$ to a Sobolev map $u_\omega\in W^{1,p}(\Omega, X_\omega )$ with the following properties:
 \begin{enumerate}
     \item\label{ultralimit:lipschitz-compatibility} If the sequence $(u_m)$ is Lipschitz bounded, then $u_\omega$ agrees with the pointwise ultralimit of $(u_m)$ almost everywhere.
     \item\label{ultralimit:postcomposition-compatibility} If $(\phi_m \colon X_m \to Y_m)$ is Lipschitz bounded and $\phi_\omega \colon X_\omega \to Y_\omega$ the pointwise ultralimit, then the sequence $( v_m = \phi_m \circ u_m )$ is $p$-bounded and $v_\omega=\phi_\omega\circ u_\omega$ almost everywhere.
     \item\label{ultralimit:Lp-compatibility} We have $$\|d_\omega(u_\omega, v_\omega)\|_{L^p(\Omega)}= \lim\nolimits_\omega \|d_m(u_m, v_m)\|_{L^p(\Omega)}$$ for $p$-bounded $( u_m \colon \Omega \to X_m )$ and $( v_m \colon \Omega \to X_m )$.
 \end{enumerate}
\end{theorem}
Uniqueness in \Cref{theorem-ultralimit-construction} is understood in the sense that if another assignment $( u_m ) \mapsto \widetilde{u}_\omega$ satisfies axioms \eqref{ultralimit:lipschitz-compatibility}-\eqref{ultralimit:Lp-compatibility}, then $u_\omega = \widetilde{u}_\omega$ almost everywhere for every $p$-bounded $(u_m)$. This allows us to talk about \emph{the ultralimit (with respect to $\omega$)}, which moreover has the following desirable properties for calculus of variations.
\begin{theorem}\label{theorem-ultralimit-trace-properties}
The ultralimit from \Cref{theorem-ultralimit-construction} satisfies the following properties.
\begin{enumerate}\setcounter{enumi}{3}
    \item\label{ultralimit:lower-semicontinuity-energy} For every $p$-bounded $(u_m)$ and every lower semicontinuous energy, it holds that $$\lim\nolimits_{\omega} E^{p}( u_m ) \geq E^p( u_\omega ).$$
    \item\label{ultralimit:lower-semicontinuity-volume} For every $p$-bounded $(u_m)$ with $p\geq n$ and every lower semi-continuous volume $$\lim\nolimits_{\omega} \mathrm{Vol}(u_m) \geq \mathrm{Vol}(u_\omega).$$
    \item\label{ultralimit:comparing-traces} The traces of $p$-bounded sequences $( u_m \colon \Omega \to X_m ), ( v_m \colon \Omega \to X_m )$ satisfy 
        \begin{align*}
            \|  d_\omega( \trace( u_\omega ), \trace( v_\omega ) ) \|_{ L^{p}( \partial \Omega ) }
            =
            \lim\nolimits_{\omega} \| d_m( \trace( u_{m} ), \trace( v_{m} ) ) \|_{ L^{p}( \partial \Omega ) }.
        \end{align*}
    \item\label{ultralimit:compatibility-of-traces} If the traces $\trace(u_m) \colon \partial \Omega \to X_m$ have continuous representatives $( \gamma_m )$ that are equi-continuous and equi-bounded, then the pointwise ultralimit $\gamma_\omega$ is a continuous representative of the trace of $u_\omega$.
\end{enumerate}
\end{theorem}
In the paper \cite{guo-wenger-2020-area-minimizing-discs-in-locally-non-compact-metric-spaces}, Guo and the second author were able to define an ultralimit of a suitable subsequence of a given $p$-bounded sequence. Passing to a subsequence, however, often poses a problem, since the resulting map takes values in the ultralimit of the \emph{subsequence} of spaces instead of the ultralimit of the \emph{full sequence} of spaces. 

Our main interest in the ultralimit construction arises from the stability of Dehn functions under ultralimits.
\subsection{Stability of Dehn functions}
In order to state our stability results, we define the Sobolev filling area of a Lipschitz curve $\gamma \colon \mathbb{S}^1 \to X$ by
\begin{align*}
    \mathrm{FillArea}_X( \gamma ) \coloneqq \inf\{ \mathrm{Area}(v) \colon v \in W^{1,2}( \mathbb{D}, X ), \, \trace(v) = \gamma \},
\end{align*}
where $\mathbb{D} \subset \mathbb{R}^2$ is the Euclidean unit disk and $\mathbb{S}^1$ its boundary. See \Cref{section-parametrized-volume} for the definition of parametrized Hausdorff area $\mathrm{Area}(v)$. Here we simply mention that if $v$ is injective and Lipschitz, then $\mathrm{Area}(v)$ is equal to the two-dimensional Hausdorff measure of the image. Moreover, in case  $v \in W^{1,2}( \mathbb{D}, X )$ and $X$ is Riemannian, $\mathrm{Area}(v)$ is obtained by integrating the Jacobian of the weak differential of $v$. The (Sobolev--)Dehn function of $X$ is 
\begin{align*}
    \delta_{X}(r) = \sup\{ \mathrm{FillArea}_X(\gamma) \colon \text{ $\gamma \colon \mathbb{S}^1 \to X$ is Lipschitz and $\ell(\gamma) \leq r$} \}
\end{align*}
for $r > 0$, where $\ell( \gamma )$ denotes the length of $\gamma$. 

It was recently proved in a paper by Lytchak, the second author, and Young \cite{Lyt:Wen:You:20} that in the setting of sequences of proper length spaces, a local quadratic isoperimetric inequality is stable under ultralimits. It was then asked by Stadler and the second author whether this remains true in the non-proper setting and for general Dehn functions; see \cite[Problem 1]{stadler-wenger-2025-isoperimetric-inequalities-vs-upper-curvature-bounds}. The following theorem resolves the problem positively.
\begin{theorem}\label{theorem-stability-dehn-function}
Let $r_0 \in (0,\infty]$ and let $\delta \colon (0,r_0) \to (0,\infty)$ be continuous from the right. If $( X_m, d_m, p_m )$ is a sequence of complete pointed length spaces with
\begin{align*}
    \delta_{X_m}(r) \leq \delta(r)
    \quad\text{for every $r \in (0, r_0)$},
\end{align*}
then the ultralimit $X_\omega$ satisfies
\begin{align*}
    \delta_{X_\omega}( r ) \leq \delta(r)
    \quad\text{for every $r \in (0,r_0)$.}
\end{align*}
\end{theorem}
A coarse version of \Cref{theorem-stability-dehn-function} will be proved in \Cref{theorem-stability-dehn-function-coarse}.

There is an analogous definition of Lipschitz--Dehn function using Lipschitz maps instead of Sobolev maps. However, we cannot expect stability properties for the Lipschitz--Dehn function since a bound on the growth does not give control on the Lipschitz constants of the fillings.

We next highlight key applications of the stability.

\subsection{Applications}
Our study of Dehn functions is partly motivated by the recent \emph{analytic} characterization of $\mathrm{CAT}(\kappa)$ spaces --- complete metric spaces with curvature bounded above by $\kappa$ in the sense of Alexandrov. Let $M^{2}_\kappa$ be the complete simply connected surface of constant sectional curvature $\kappa$ and let $D_\kappa$ be its diameter; when $\kappa \neq 0$, $M^{2}_\kappa$ is isometric to a rescaled sphere or a hyperbolic plane. Let $\delta_\kappa$ be the Dehn function of $M^{2}_\kappa$, see \cite[Section 4.1]{stadler-wenger-2025-isoperimetric-inequalities-vs-upper-curvature-bounds} for a precise formula. When $\kappa = 0$, it holds that $\delta_{0}(r) = r^2/4\pi$ for $r > 0$ but always $\lim_{ r \to 0^{+} }\delta_{\kappa}(r)/r^2 = 1/(4\pi)$. We recall that if $X$ is $\mathrm{CAT}(\kappa)$, then the Dehn function of $X$ is bounded from above by $\delta_\kappa$ by Reshetnyak's majorization principle \cite{reshetnyak-1968-non-expansive-maps-in-a-space-of-curvature-no-greater-than-K}. The converse is known in proper length spaces by a theorem due to Lytchak and the second author \cite{Lyt:Wen:18:CAT}. More recently, Stadler and the second author \cite{stadler-wenger-2025-isoperimetric-inequalities-vs-upper-curvature-bounds} extended this to the non-proper setting as follows.
\begin{theorem}[{\cite[Theorem A]{stadler-wenger-2025-isoperimetric-inequalities-vs-upper-curvature-bounds}}]\label{theorem-local-cat(kappa)}
Let $\kappa \in \mathbb{R}$ and $r_0 \in (0, 2D_\kappa]$. Let $X$ be a complete length space and suppose that $\delta_{X}(r) \leq \delta_{\kappa}(r)$ for $r \in (0,r_0)$. Then every closed ball of radius $r_0/4$ of $X$ is geodesic and $\mathrm{CAT}(\kappa)$. Moreover, if $r_0 = 2D_\kappa$, then $X$ is a $\mathrm{CAT}(\kappa)$ space.
\end{theorem}
    The general strategy of \cite{stadler-wenger-2025-isoperimetric-inequalities-vs-upper-curvature-bounds} is similar to that of \cite{Lyt:Wen:18:CAT}. However, it involves a delicate analysis of the intrinsic structure of a solution to a Plateau problem in the ultracompletion of $X$. Using the stability result, \Cref{theorem-stability-dehn-function}, we can give a proof that is considerably simpler by employing the strategy  of \cite{Lyt:Wen:18:CAT} directly.

%

Our second application concerns Gromov hyperbolic spaces. In \cite{gromov-1987-hyperbolic-groups}, Gromov famously showed that spaces with a subquadratic isoperimetric inequality are Gromov hyperbolic and thus have a linear isoperimetric inequality. In fact, it suffices to have a quadratic isoperimetric inequality with a sufficiently small constant. The second author proved a sharp version of Gromov's result in \cite{wenger-2008-gromov-hyperbolic-spaces-and-the-sharp-isoperimetric-constant} using the theory of integral currents. We establish the following version using our stability results.
\begin{theorem}\label{theorem-gromov's-theorem-introduction}
Suppose that $X$ is a complete geodesic metric space with
\begin{align*}
    \limsup_{ r \rightarrow \infty } \frac{ \delta_{X}(r) }{ r^2 } < \frac{1}{4\pi}.
\end{align*}
Then $X$ is Gromov hyperbolic.
\end{theorem}
In the setting of proper geodesic metric spaces, a proof of \Cref{theorem-gromov's-theorem-introduction} was given in \cite{Lyt:Wen:You:20}. Our proof parallels this proof and relies, in addition, on the stability of Dehn functions, \Cref{theorem-stability-dehn-function}, or rather the coarse version, \Cref{theorem-stability-dehn-function-coarse}. 


\subsection{Outlines of the proofs and additional results}
We begin by highlighting some of the ideas in the proof of the ultralimit theorems. 
The central notion we introduce and study in this context is that of a Lusin--Lipschitz admissible sequence, cf.~\Cref{section-lusin-lipschitz-admissible-sequences}. Roughly speaking, this is a sequence of Lipschitz maps $f^k$ from $\Omega$ into a metric space $X$ which agree on larger and larger subsets of $\Omega$, eventually exhausting $\Omega$ in measure. Lusin--Lipschitz admissible sequences have natural limits because, by definition, $(f^k(z))$ is eventually constant for almost every $z\in\Omega$.  Now, let $(u_m)$ be a $p$-bounded sequence of Sobolev maps $u_m\colon\Omega\to X_m$ as in \Cref{theorem-ultralimit-construction}. We first embed $X_m$ isometrically into a larger space $X'_m$ with suitable Lipschitz extension properties. This allows us to show that each $u_m$ is the limit of a Lusin--Lipschitz admissible sequence $(u_m^k)$ of maps $u_m^k\colon\Omega\to X'_m$ in such a way that, for each $k$, the sequence $(u_m^k)$ is Lipschitz bounded. In particular, for each $k$, the pointwise ultralimit $u_\omega^k$ is a Lipschitz map into the ultralimit $X'_\omega$ of the pointed sequence $(X'_m)$. The crucial point we prove and use in \Cref{section-the-ultralimit-construction} is that the sequence $(u_\omega^k)$ is again Lusin--Lipschitz admissible and, for each $k \in \N$, $u_\omega^k$ maps a large subset of $\Omega$ into $X_\omega$. This enables us to define its limit $u_\omega$ and to show that $u_\omega$ has image in the subset $X_\omega$ of $X'_\omega$, after possibly modifying $u_\omega$ on a negligible subset. The map $u_\omega$ is our candidate for the ultralimit of the $p$-bounded sequence $(u_m)$.

In order to prove that $u_\omega$ is Sobolev and has the properties stated in \Cref{theorem-ultralimit-construction,theorem-ultralimit-trace-properties}, we use a \emph{geometric realization} of the ultralimit. Namely, we establish the following theorem akin to Gromov's embedding theorem \cite[Section 6]{gromov-1981-groups-of-polynomial-growth-and-expanding-maps}.

\begin{restatable}{theorem}{GromovEmbedding}
\label{thm:embedding-mapping-packages}
 Let $(Y^k)_{k\in I}$ be a countable collection of separable metric spaces and $(X_m, d_m, p_m)$ a sequence of complete pointed metric spaces. Let $f_m^k\colon Y^k\to X_m$ be continuous maps for $m\in\N$ and $k\in I$, and suppose that the sequence $( f_m^k )$ is equi-continuous and equi-bounded for each $k\in I$. Then:
 \begin{enumerate}
  \item there exist a complete metric space $Z$, a subsequence $(m_l)$, and isometric embeddings $\iota_l\colon X_{m_l}\to Z$ such that, for each $k\in I$, the sequence $(\iota_l\circ f_{m_l}^k)$ converges to a continuous map $f_\infty^{k} \colon Y^k \to Z$ pointwise everywhere and uniformly on compact sets;
   \item there exists an isometric embedding $\eta\colon C\to X_\omega$, where $C=\overline{\bigcup_{k\in I} f_\infty^k(Y^k)}$, such that $$f_\omega^k = \eta\circ f_\infty^k$$ for every $k\in I$.
 \end{enumerate}
 Moreover, given subsets $\N\supset \mathcal{F}_1\supset \mathcal{F}_2\supset\dots$ satisfying $\omega(\mathcal{F}_m) = 1$ for all $m$, the subsequence $(m_l)$ can be chosen so that $m_l\in \mathcal{F}_{N_l}$ for a strictly increasing sequence $( N_l )$.
\end{restatable} 
The flexibility provided by this result allows the independent study of the properties stated in \Cref{theorem-ultralimit-construction,theorem-ultralimit-trace-properties}. 

Next, we formulate a theorem underlying the stability of Dehn functions, which can be understood as lower semicontinuity of filling areas under ultralimits.

\begin{theorem}\label{thm:stability-introduction}
Suppose that $( X_m, d_m, p_m )$ is a sequence of complete pointed length spaces and $X_\omega$ the ultralimit. If $( \gamma_m \colon \mathbb{S}^1 \to X_m )$ is a Lipschitz bounded sequence with $\gamma = \lim\nolimits_{\omega} \gamma_m$ and $\lim\nolimits_{\omega} \mathrm{FillArea}_{X_m}( \gamma_m ) < \infty$, then there exists $u \in W^{1,2}( \mathbb{D}, X_\omega )$ with
\begin{align*}
    \trace(u) = \gamma
    \quad\text{and}\quad
    \mathrm{Area}( u )
    \leq
    \lim\nolimits_\omega \mathrm{FillArea}_{X_m}( \gamma_m ).
\end{align*}
\end{theorem}
The theorem implies, in particular, that $$\mathrm{FillArea}_{X_\omega}(\gamma)\leq \lim\nolimits_\omega \mathrm{FillArea}_{X_m}( \gamma_m ).$$
The proof builds upon the ultralimit construction of Sobolev maps and uses, in addition, a mapping cylinder construction similar to \cite{stadler-the-structure-of-minimal-surfaces-in-cat-0-spaces,Creu:22}. A crucial point which we exploit is that the mapping cylinder construction is robust under taking ultralimits.

\subsection{Final remark and structure of the paper}\label{section-structure-of-the-paper}
In convex geometry there are several natural definitions $\mu$ of area in two-dimensional normed spaces, among which the Holmes--Thompson area $\mu^{ht}$, the Benson or Gromov $\text{mass}^{*}$ area $\mu^{*}$, Ivanov's inscribed Riemannian area $\mu^{i}$, as well as the Hausdorff or Benson-Hausdorff area $\mu^{bh}$; see \cite{alvarez-thompson-2004-volumes-on-normed-and-finsler-spaces} and \Cref{section-lower-semi-continuous-volumes} below. Every definition $\mu$ of area gives rise to a parametrized area $\mathrm{Area}^\mu(u)$ for Sobolev maps with values in a metric space, and the parametrized Hausdorff area $\mathrm{Area}(u)$ used in the subsections above is the one corresponding to the Benson-Hausdorff definition $\mu^{bh}$. For simplicity, we stated \Cref{theorem-stability-dehn-function,theorem-local-cat(kappa),theorem-gromov's-theorem-introduction,thm:stability-introduction} for the parametrized Hausdorff area but they can be generalized to other definitions of area as well. Indeed, \Cref{theorem-stability-dehn-function,thm:stability-introduction} hold for any lower semi-continuous area; see \Cref{thm:stability,theorem-stability-dehn-function-coarse}.  \Cref{theorem-local-cat(kappa)} extends to $\mu^*$ and $\mu^{i}$. In fact, the statement holds for any lower semicontinuous area for which the isoperimetric constant is $1/(4\pi)$ and inner product spaces are the only spaces realizing the isoperimetric constant. The latter assumption is necessary since the conclusion of the theorem fails for the Holmes--Thompson area already in normed planes and $\kappa = 0$.
Finally, we also formulate a version of \Cref{theorem-gromov's-theorem-introduction} for all lower semi-continuous area $\mu$, cf. \Cref{theorem-gromov's-theorem}, where the constant $1/(4\pi)$ needs to be replaced by the \emph{isoperimetric constant} of the $\mu$ area; see \eqref{equation-the-isoperimetric-constant}.

\medskip

We end with outlining the structure of the paper. \Cref{section-preliminaries} establishes notation and reviews some basic concepts. In \Cref{sec:sobolevmaps}, we recall some necessary elements from Sobolev theory, including Haj{\l}asz gradients and various notions of area. The embedding theorem, \Cref{thm:embedding-mapping-packages}, is proved in \Cref{section-mapping-package}. We establish important properties of Lusin--Lipschitz admissible sequences in \Cref{section-lusin-lipschitz-admissible-sequences} and use them in \Cref{section-the-ultralimit-construction} to prove our ultralimit theorems, \Cref{theorem-ultralimit-construction,theorem-ultralimit-trace-properties}. \Cref{section-stability-results} contains the proofs of \Cref{thm:stability-introduction,theorem-stability-dehn-function}; in fact, we establish somewhat more general versions of these results. Finally, the applications are proved in \Cref{section-applications}.

\subsection{Acknowledgements}
Both authors were supported by Swiss National Science Foundation grant 212867. T.I. was also supported by the Research Council of Finland, project number 332671. The authors thank Stephan Stadler for fruitful discussions.

\section{Preliminaries}\label{section-preliminaries}
We give a quick overview of our basic notation. We need the notion of a \emph{pointed metric space} which is a triple $( X, d_X, p_X )$ such that $d_X \colon X \times X \rightarrow [0,\infty)$ is a distance on $X$ and $p_X \in X$ is a point. The open and closed metric balls centered at $y \in X$ and of radius $r > 0$ are denoted by $B_X( y, r )$ and $\overline{B}_X( y, r )$, respectively. Typically it will be clear from the context which space the ball is being considered in and in such a case we drop the subscript $X$. We use the same convention for the distance as well. We will be mainly working with \emph{complete} metric spaces. Given a non-empty set $A \subset X$ in a metric space, $\overline{A}$ will denote the \emph{closure} of $A$.

Lipschitz mappings between metric spaces play an important role for us so we recall the definition. Given a pair of metric spaces $X$ and $Y$, a map $f \colon X \rightarrow Y$ is \emph{Lipschitz} if
\begin{align*}
    \sup_{ x \neq x' } \frac{ d(f(x),f(x')) }{ d(x,x') } < \infty.
\end{align*}
The supremum is the \emph{Lipschitz constant} of $f$ and we denote it by $\LIP(f)$. We say that $f$ is $L$-Lipschitz if $\LIP(f) \leq L$. We say that $f$ is \emph{$L$-bi-Lipschitz} if
\begin{align*}
    L^{-1} d(x,y) \leq d(f(x),f(y)) \leq Ld(x,y)
    \quad\text{for every $x, y \in X$.}
\end{align*}
If $L = 1$, we say that $f$ is \emph{isometric} or an \emph{isometric embedding}.

\subsection{Curves and length distances}
A \emph{curve} is a continuous map $\gamma \colon I \rightarrow X$ where $I$ is an interval or the \emph{unit circle} $\mathbb{S}^1$. If $I$ is an interval, then the length of $\gamma$ is defined by
\begin{align*}
    \ell( \gamma )
    =
    \sup\left\{
        \sum_{i=0}^{k-1} d( \gamma(t_i), \gamma(t_{i+1}) )
        \colon t_i \in I
        \,\text{and}\,
        t_0 < t_1 < \dots < t_k
    \right\}
\end{align*}
and an analogous definition applies in the case $I = \mathbb{S}^1$. A \emph{Jordan curve} $\Gamma$ is the image of a topological embedding $\gamma \colon \mathbb{S}^1 \rightarrow X$. A curve $\gamma \colon \mathbb{S}^1 \rightarrow X$ is a \emph{weakly monotone parametrization of $\Gamma$} if it a uniform limit of homeomorphisms $\mathbb{S}^1 \rightarrow \Gamma$. 

A metric space $X$ is a \emph{length space} if, for every $x, y \in X$,
\begin{align*}
    d(x,y) = \inf\{ \ell(\gamma) \colon \gamma \colon [a,b] \to X, \gamma(a) = x, \gamma(b) = y \}.
\end{align*}
A curve $\gamma \colon [a,b] \to X$ is a \emph{geodesic} if $d( \gamma(a), \gamma(b) ) = \ell( \gamma )$, and a metric space $X$ is \emph{geodesic} if every pair $x, y \in X$ can be joined with a geodesic.

\subsection{Hausdorff measures}
We consider the $n$-dimensional Hausdorff measure on a metric space $X$. We recall that for a set $A \subset X$ we have
\begin{align*}
    \mathcal{H}^{n}( A ) = \sup_{ \delta > 0 } \mathcal{H}^{n}_{\delta}(A),
\end{align*}
where
\begin{align*}
    \mathcal{H}^{n}_{\delta}(A)
    =
    \frac{ \omega_n }{ 2^n }
    \inf\left\{
        \sum_{ i } \diam(E_i)^{n}
        \colon
        \text{ $A \subset \bigcup E_i$ and $\diam E_i < \delta$ for $i \in \N$ }
    \right\}.
\end{align*}
Here $\omega_n$ is the Lebesgue measure of the $n$-dimensional Euclidean unit ball. This normalization guarantees, in particular, that $\mathcal{H}^n$ coincides with the $n$-dimensional Lebesgue measure $\lm^{n}$ in the $n$-dimensional Euclidean space $\mathbb{R}^n$.

\subsection{Non-principal ultrafilters}\label{section-ultrafilters}
We review standard definitions related to a \emph{non-principal ultrafilter} $\omega$ on $\N$. First, recall that a non-principal ultrafilter is a finitely additive measure $\omega \colon \mathcal{P}( \mathbb{N} ) \rightarrow \{0,1\}$ on the power set $\mathcal{P}( \mathbb{N} )$ of the natural numbers $\mathbb{N}$ giving zero measure to finite sets and measure one to $\N$.

Given a bounded sequence $( r_m )$ of real numbers, we say that $t \in \mathbb{R}$ is the \emph{ultralimit} if
\begin{equation*}
    \omega\left(
    \left\{
        m \in \mathbb{N}
        \colon
        | r_m - t |
        <
        \varepsilon
    \right\}
    \right)
    =
    1
    \quad\text{for every $\varepsilon > 0$.}
\end{equation*}
The ultralimit of a bounded sequence always exists and is unique, justifying the notation $\lim_\omega r_m$. For bounded sequences in $\R^n$ we similarly may take componentwise ultralimits. This allows us to take ultralimits of a sequence $(E_m)$ of non-empty subsets of $\Omega$ when $\Omega \subset \R^n$ is bounded: We denote
\begin{align*}
    E_\omega = \{ \lim\nolimits_{\omega} x_m \in \R^n \colon \text{$x_m \in E_m$ for every $m \in \N$} \}
\end{align*}
and call $E_\omega$ the ultralimit of $(E_m)$.

We will later need the following semicontinuity result for the measure under ultralimits.
\begin{lemma}\label{lem:bound-ultralimit-set}
Let $(E_m)$ be a sequence of non-empty measurable subsets of a bounded open set $\Omega \subset \R^n$. Then the ultralimit $E_\omega$ is a non-empty compact subset of $\overline{\Omega}$ and satisfies $$\lm^n(\Omega\cap E_\omega) \geq \lim\nolimits_\omega \lm^n(E_m).$$
\end{lemma}

\begin{proof}
Clearly $E_\omega$ is a non-empty compact subset of $\overline{\Omega}$ so it suffices to prove the measure lower bound for $\Omega \cap E_\omega$.
%
%
Observe that the sequence of characteristic functions $g_m = \chi_{ E_m }$ forms a bounded sequence in $L^{\infty}( \mathbb{R}^n )$ and as $L^{\infty}( \mathbb{R}^n )$ is the dual of $L^{1}( \mathbb{R}^n )$, there exists $g_\omega \in L^{\infty}( \mathbb{R}^n )$ such that
\begin{align*}
    \int_{ \mathbb{R}^n } h g_\omega \,dz = \lim\nolimits_{\omega} \int_{ \mathbb{R}^n } h g_m \,dz
    \quad\text{for every $h \in L^{1}( \mathbb{R}^n )$.}
\end{align*}
In particular, if we apply this equality to nonnegative $h$, it immediately follows that $0 \leq g_\omega \leq 1$ almost everywhere. We conclude that
\begin{align*}
    \lm^n( \Omega \cap \left\{ g_\omega > 0 \right\} )
    \geq
    \lim\nolimits_{\omega} \int_{ \Omega } g_m \,dz
    =
    \lim\nolimits_{\omega} \lm^n( E_m ).
\end{align*}
If $x \in \Omega \cap \left\{ g_\omega > 0 \right\}$ is a Lebesgue density point, we claim, in fact, that $x \in E_\omega$. This will finish the proof. Let
\begin{align*}
    \mathcal{G}_k = \{ m \in \N \colon \lm^{n}( E_m \cap B( x, k^{-1} ) ) > 0 \} \cap [k,\infty).
\end{align*}
It holds that $\mathcal{G}_{k} \supset \mathcal{G}_{k+1}$ and $\omega( \mathcal{G}_k ) = 1$ for $k \in \N$ by construction. To finish, let $x_m \in E_m$ for $m \in \N \setminus \mathcal{G}_1$ and $x_m \in E_m \cap B( x, k^{-1} )$ if $m \in \mathcal{G}_{k} \setminus \mathcal{G}_{k+1}$. Since $x = \lim\nolimits_{\omega} x_m$, the proof is complete.
\end{proof}

\subsection{Ultralimits of spaces}\label{section-ultralimits-of-spaces}
In this subsection, we consider a sequence $( X_m, d_m, p_m )$ of pointed metric spaces. A sequence $(x_m)$, for $x_m \in X_m$, is \emph{bounded} if
\begin{align*}
    \sup_{ m \in \mathbb{N} } d_m( x_m, p_m ) < \infty.
\end{align*}
We say that two bounded sequences $(x_m)$ and $(y_m)$ are \emph{equivalent} if
\begin{equation*}
    \lim\nolimits_{\omega} d_m( x_m, y_m ) = 0.
\end{equation*}
The equivalence class of a bounded sequence $(x_m)$ is denoted by $[ x_m ]$ and the collection of such equivalence classes we denote by $X_\omega$. We refer to $[x_m]$ as the \emph{ultralimit} of a bounded sequence $( x_m )$. We at times use the notation $\lim\nolimits_{\omega} x_m$ to refer to $[x_m]$ as well.

Observe that
\begin{equation*}
    d_\omega( [ x_m ], [ y_m ] )
    \coloneqq
    \lim\nolimits_{\omega} d_{m}( x_m, y_m )
\end{equation*}
defines a distance function on $X_\omega$. We also denote $p_\omega \coloneqq [ p_m ]$. We say that the triple $( X_\omega, d_\omega, p_\omega )$ is the \emph{$\omega$-ultralimit}, or the \emph{ultralimit}, of the pointed metric spaces $( X_m, d_m, p_m )$. We often use the shorthand notation $X_\omega$ to refer to $( X_\omega, d_\omega, p_\omega )$. We note that $X_\omega$ is \emph{always} complete. Moreover, if $( X_m, d_m, p_m )$ are length spaces, then $X_\omega$ is geodesic. When we consider the constant sequence $( X, d, p )$, the ultralimit $X_\omega$ is called an \emph{ultracompletion} of $X$. In this case, we have the canonical isometric embedding $X \xhookrightarrow{} X_\omega$ where $x \mapsto [x]$.

In case $( X_m, d_m, p_m ) = ( X, \lambda_m d, p_m )$ for $\lambda_m \rightarrow 0^{+}$, the ultralimit $X_\omega$ is called an \emph{asymptotic cone} of $X$. If instead $p_m = p$ and $\lambda_m \rightarrow \infty$, then $X_\omega$ is a \emph{tangent (cone)} of $X$ (at the point $p$).

We note that if pointed metric spaces $( X_m, d_m, p_m )$ converge to a proper $( Z, d, p )$ in pointed Gromov--Hausdorff topology, then the ultralimit $X_\omega$ is isometric to $( Z, d, p )$, cf. \cite[Theorem 5.16]{alexander-kapovitch-petrunin-2024-alexandrov-geometry-foundations}.
%

\subsection{Ultralimits of mappings}\label{section-ultralimits-of-mappings}
We also consider ultralimits of mappings. Consider sequences of pointed metric spaces $( X_m, d_m, p_m )$ and $( Y_m, d_m, q_m )$.
\begin{enumerate}
    \item A sequence $( f_m \colon (X_m,d_m,p_m) \to ( Y_m, d_m, q_m ) )$ is \emph{equi-continuous on bounded sets} if for every $R > 0$ and $\varepsilon > 0$, there exists $\delta > 0$ such that $d_m( x, y ) < \delta$ implies $d_m( f_m(x), f_m(y) ) < \varepsilon$ whenever $x, y \in B_{X_m}( p_m, R )$ and $m \in \N$.  In case we may take $R = \infty$, the sequence is \emph{equi-continuous}.
    \item A sequence $( f_m \colon (X_m,d_m,p_m) \to ( Y_m, d_m, q_m ) )$ is \emph{equi-bounded on bounded sets} if for every $R > 0$, there exists $\lambda > 0$ such that $f_m( \overline{B}_{X_m}( p_m, R ) ) \subset \overline{B}_{Y_m}( q_m, \lambda )$ for every $m \in \N$. In case we may take $R = \infty$, the sequence is \emph{equi-bounded}.
\end{enumerate}
Now, in case a sequence $( f_m \colon (X_m,d_m,p_m) \to ( Y_m, d_m, q_m ) )$ is equi-continuous and equi-bounded on bounded sets, there exists a \emph{pointwise ultralimit} $f_\omega \colon X_\omega \to Y_\omega$ defined by $f_\omega( \lim_\omega x_m ) = \lim_\omega f_m(x_m)$ (or equivalently $f_\omega( [x_m] ) = [f_m(x_m)]$). The construction of $f_\omega$ implies that $f_\omega$ maps bounded sets to bounded sets and it has the same modulus of continuity as $( f_m )$. Moreover, if $( X_m, d_m, p_m ) = ( X, d, p )$, then $f_\omega(x) = [f_m(x)]$ holds for $x \in X$ under the canonical identification $X \subset X_\omega$ given by $x \mapsto [x]$.

\subsection{Injective metric spaces}\label{section-injective-metric-spaces}
We will need injective metric spaces in the sequel. Recall that a metric space $Y$ is \emph{injective} if for every subset of a metric space $E \subset Z$, any $1$-Lipschitz map $E \to Y$ extends to a $1$-Lipschitz map $Z \to Y$. In particular, every injective metric space is complete. Also recall that via the Kuratowski embedding, every metric space $X$ embeds into the injective Banach space $\ell^{\infty}(X)$ isometrically; here $\ell^{\infty}(X)$ is the space of bounded functions $X \to \mathbb{R}$ equipped with the supremum norm. The injectivity of $\ell^{\infty}(X)$ is easy to see using componentwise McShane extension. Recall that the $\lambda$-Lipschitz McShane extension $\hat{f} \colon Z \to \R$ of a $\lambda$-Lipschitz $f \colon Z \to \R$ is defined by $\hat{f}(z) = \inf\{ f(y) + \lambda d(y,z) \colon y \in E \}$. Then the truncated McShane extension $\bar{f} \colon Z \to \R$, where $\bar{f}(z) = \max\{ \sup_{y \in E} |f(y)|, \hat{f}(z) \}$, leads to a $\lambda$-Lipschitz extension $\bar{f} \colon Z \to \R$ with $\sup_{z \in Z} |\bar{f}(z)| = \sup_{z \in E} |f(z)| \in [0,\infty]$. This implies the standard fact that any closed ball in any injective space is injective.

An \emph{injective hull} $Y$ of a metric space $X$ is an injective metric space equipped with an isometric embedding $e \colon X \to Y$ that is minimal in an appropriate sense among injective spaces containing $X$. Injective hulls were constructed by Isbell in \cite{isbell-1964-six-theorems-about-injective-metric-spaces}.

\section{Metric valued Sobolev mappings}\label{sec:sobolevmaps}
We briefly review metric valued Sobolev theory for domains in Euclidean spaces. There exist several equivalent definitions and we refer to the monograph \cite{Hei:Kos:Sha:Ty:15} and the recent manuscripts \cite{Lyt:Wen:17:areamini,Creu:Evs:24} for details. See also \cite{Kor:Sch:93,Haj:96,Sha:00,Hei:Kos:Sha:Ty:01}. We will use the approach due to Reshetnyak \cite{Res:97} based on post composition by Lipschitz functions.

\subsection{Sobolev spaces due to Reshetnyak}
Let $X$ be a complete metric space, $p \in (1,\infty)$ and let $\Omega \subset \mathbb{R}^n$ be a bounded open set. We denote by $L^{p}( \Omega, X )$ the set of measurable and essentially separably valued $u \colon \Omega \to X$ such that $d(u,p_X) \in L^{p}( \Omega )$ for some (and thus any) $p_X \in X$, the classical Lebesgue space of $p$-integrable functions. The distance on $L^p( \Omega, X )$ is defined by
\begin{align*}
    \| d(u,v) \|_{L^{p}(\Omega)} = \left( \int_{\Omega} ( d(u,v) )^p \,dz \right)^{ \frac{1}{p} },
    \quad\text{for $u,v \in L^{p}( \Omega, X )$.}
\end{align*}
Since $X$ is complete, the space $L^p( \Omega, X )$ becomes a complete metric space.

A map $u \in L^{p}( \Omega, X )$ belongs to the Sobolev space $u \in W^{1,p}( \Omega, X )$ if there exists an $L^{p}( \Omega )$-integrable $\rho \colon \Omega \rightarrow \R$ such that for every $1$-Lipschitz $\varphi \colon X \rightarrow \mathbb{R}$, it holds that $x \mapsto \varphi \circ u(x)$ belongs to $W^{1,p}( \Omega )$ in the classical sense and the distributional gradient satisfies $|\nabla( \varphi \circ u ) | \leq \rho$ almost everywhere. The \emph{Reshetnyak $p$-energy} of $u \in W^{1,p}( \Omega, X )$ is defined by
\begin{align*}
    E_{+}^{p}( u ) = \inf\left\{ \| \rho \|_{ L^{p}(\Omega) }^p \colon \rho \text{ is as above} \right\}.
\end{align*}
The $p$-energy is extended as $\infty$ to all of $L^{p}( \Omega, X )$ after which $E_{+}^{p} \colon L^{p}( \Omega, X ) \rightarrow [0,\infty]$ becomes lower semi-continuous with respect to $L^{p}( \Omega, X )$-convergence. In case $E_+^{p}(u) < \infty$, standard arguments show that the infimum above is a minimum and the minimizer is unique almost everywhere. We denote the minimizer by $|\nabla u|$. For real-valued functions, $|\nabla u|$ is the modulus of the weak differential.

\subsection{Lower semi-continuous energies}\label{section-lower-semi-continuous-energies}
In the sequel, we consider lower semi-continuous energies. We say that an energy functional $E^p$ is \emph{lower semi-continuous} if for every complete metric space $X$ and bounded Lipschitz domain $\Omega \subset \mathbb{R}^n$, any $p$-bounded $( u_m \colon \Omega \to X )$ with $\lim_{ m \to \infty } \| d( u_m, u ) \|_{ L^{p}(\Omega) } = 0$ satisfies
\begin{align*}
    \liminf_{ m \to \infty } E^{p}( u_m ) \geq E^{p}( u ).
\end{align*}
The standard energies such as the Korevaar--Schoen energy, the Reshetnayk energy, and the Haj{\l}asz energy are lower semi-continuous in this sense, cf. \cite[Section 5]{Lyt:Wen:17:areamini}. We also recall that the Newton--Sobolev energy coincides with the Reshetnyak energy on all domains $\Omega \subset \mathbb{R}^n$, and all of these approaches define comparable energies for bounded Lipschitz domains $\Omega \subset \mathbb{R}^n$, cf. \cite[Theorems 6.1.17 and 7.1.20 and Section 10]{Hei:Kos:Sha:Ty:15}. 

\subsection{Haj{\l}asz gradients}
When $\Omega \subset \R^n$ is a bounded Lipschitz domain, Reshetnyak's definition of a Sobolev map above is equivalent to that of a Haj{\l}asz--Sobolev map (see \cite{Haj:96} or \cite[Section 10]{Hei:Kos:Sha:Ty:15}) and thus, for every $u \in W^{1,p}( \Omega, X )$, there exists a nonnegative function $g \in L^{p}( \Omega )$ and a negligible set $N \subset \Omega$ such that
\begin{align}\label{eq:hajlaszgradient}
    d(u(x),u(y)) \leq ( g(x) + g(y) ) d(x,y) \quad\text{for every $x, y \in \Omega \setminus N$.}
\end{align}
We call $g$ a \emph{Haj{\l}asz gradient} of $u$. For convenience sake, we will assume that our Haj{\l}asz gradients are set to infinity in the set where \eqref{eq:hajlaszgradient} would fail otherwise. Thus \eqref{eq:hajlaszgradient} holds for \emph{every} pair $x,y \in \Omega$. It holds that for $u \in W^{1,p}( \Omega, X )$, there exists a Haj{\l}asz gradient $g$ of $u$ satisfying
\begin{equation}\label{eq:normestimate:hajlasz:sobolev}
    \| g \|_{ L^{p}(\Omega) }^p \leq C E^{p}_{+}( u )
    \quad\text{for every $u \in W^{1,p}( \Omega, X )$}
\end{equation}
for a constant $C$ depending only on $\Omega$ and $p$.

\subsection{Trace theory}\label{section-trace-theory}
When $\Omega \subset \mathbb{R}^n$ is a bounded Lipschitz domain, the \emph{trace} of $u \in W^{1,p}( \Omega, X )$ is a mapping $h \in L^{p}( \partial \Omega, X )$, the definition of $L^{p}( \partial \Omega, X )$ being analogous to $L^{p}( \Omega, X )$ above, such that
\begin{equation*}
    \limsup_{ r \rightarrow 0^+ }
    \aint{ B(x,r) \cap \Omega } d( u(z), h(x) ) \,dz
    =
    0
    \quad\text{for $\mathcal{H}^{n-1}$-a.e. $x \in \partial \Omega$,}
\end{equation*}
where $\aint{ E } h \,dz$ refers to the integral average over a set $E$. It is clear that if $h$ exists, it is unique and thus we use the notation $\mathrm{tr}(u) = h$. The trace above is equivalent to the Korevaar--Schoen definition \cite[Section 1.12]{Kor:Sch:93}, cf. \cite[Section 3]{guo-huang-wang-xu-2022-p-harmonic-mappings-between-metric-spaces}. We state some basic properties of traces. We note that the uniqueness implies the composition rules: $\Phi \circ \trace(u)(x) = \trace( \Phi \circ u )(x)$ and $d( \trace(u)(x), \trace(v)(x) ) = \trace( d(u,v) )(x)$ for almost every $x \in \partial \Omega$ for $u,v \in W^{1,p}( \Omega, X )$ and Lipschitz $\Phi \colon X \to Y$. The trace $\trace \colon W^{1,p}( \Omega, X ) \to L^{p}( \partial \Omega, X )$ for $X = ( X, d, p_X )$ is bounded in the sense that
\begin{equation*}
    \| d( \trace(u), p_X ) \|_{ L^{p}(\Omega) } \leq C\left( \| d( u, p_X ) \|_{L^{p}(\Omega)}^{p} + E_{+}^{p}(u) \right)^{1/p} 
\end{equation*}
for some $C = C(p, \Omega)$. Finally,
\begin{equation*}
    \lim_{m\to\infty}\| d_m( \trace(u_m), \trace(v_m) ) \|_{ L^{p}(\partial \Omega) } = 0
\end{equation*}
if $\lim_{m \to \infty} \| d_m( u_m, v_m ) \|_{ L^{p}(\Omega) } = 0$ and $( u_m \colon \Omega \to X_m )$ and $( v_m \colon \Omega \to X_m )$ are $p$-bounded. This follows from the $p$-boundedness of $( d(u_m,v_m) )$ and weak continuity of $\trace\colon W^{1,p}(\Omega) \to L^{p}( \partial \Omega )$.

\subsection{The Hausdorff volume and the area formula}\label{section-parametrized-volume}
We next define a notion of parametrized Hausdorff volume for Sobolev mappings. To this end, let $u \in W^{1,p}( \Omega, X )$ for $p \in (1,\infty)$. Such a $u$ has an \emph{approximate metric derivative} at almost every $z \in \Omega$ in the following sense; see \cite{Karm:07,Lyt:Wen:17:areamini}. There exists a unique seminorm in $\mathbb{R}^n$, denoted $\mathrm{md}_z(u)$, such that
\begin{align*}
    \mathrm{ap}\lim_{ z ' \rightarrow z }
    \frac{ d( u(z), u(z') ) - \mathrm{md}_z(u)( z'-z ) }{ |z'-z| }
    =
    0
\end{align*}
where $| \cdot|$ is the Euclidean norm and $\mathrm{ap}\lim$ refers to the approximate limit; see \cite[2.9.11]{Fed:69}. 
\begin{definition}\label{def:parametrizedvolume}
When $p \in [n,\infty)$ and $u \in W^{1,p}( \Omega, X )$, the parametrized Hausdorff volume of a map $u \in W^{1,p}( \Omega, X)$ is
\begin{align*}
    \mathrm{Vol}( u ) = \int_{ \Omega } \mathbf{J}( \mathrm{md}_z(u)) \,dz,
\end{align*}
where the Jacobian $\mathbf{J}(s)$ of a seminorm $s$ is the $n$-dimensional Hausdorff measure of the Euclidean unit cube in $( \mathbb{R}^n, s )$ if $s$ is a norm and zero otherwise. When the ambient dimension $n$ is two, we instead use the notation $\mathrm{Area}(u)$ and refer to it as parametrized Hausdorff area.
\end{definition}
Recall that the $n$-dimensional Hausdorff measure is normalized so that the $n$-dimensional Hausdorff measure of the standard Euclidean norm and Lebesgue measure coincide. Thus an equivalent definition of the Jacobian is $\omega_n/ \lm^{n}( \{ s \leq 1 \} )$ where $\omega_n$ is the Lebesgue measure of the $n$-dimensional ball and $\{ s \leq 1 \}$ is the unit ball defined by the norm $s$.

If $u \in W^{1,p}( \Omega, X )$ satisfies Lusin's Property ($N$), e.g. when $p > n$ \cite{Vod:00}, and thus sends sets of Lebesgue measure zero to sets of Hausdorff $n$-dimensional measure zero, then
\begin{align*}
    \mathrm{Vol}(u) = \int_{X} \sharp( \left\{ z \colon u(z) =x \right\} ) \,d\mathcal{H}^n(x)
\end{align*}
by the area formula \cite[Theorem 3.2]{Karm:07}. If $u$ is also injective, it holds that $\mathrm{Vol}(u) = \mathcal{H}^{n}( u(\Omega) )$. Nevertheless, even if $p = n$, there exists a set $N \subset \Omega$ of Lebesgue measure zero such that $u|_{ \Omega \setminus N }$ satisfies Lusin's Condition ($N$) and
\begin{align}\label{eq:areaformula:Sobolev}
    \int_{E} \mathbf{J}( \mathrm{md}_z(u) ) \,dz
    =
    \int_{X \setminus u(N)} \sharp( \left\{ z \in E \colon u(z) = x \right\} ) \,d\mathcal{H}^n(x)
\end{align}
for Lebesgue measurable $E \subset \Omega$. It follows from \eqref{eq:areaformula:Sobolev} that $u|_{ \{ \mathbf{J}( \mathrm{md}(u)) > 0 \} \setminus N}$ satisfies Lusin's Conditions $(N)$ and $(N^{-1})$.

\subsection{Lower semi-continuous volumes}\label{section-lower-semi-continuous-volumes}
Important examples of volumes arise in convex geometry \cite{alvarez-thompson-2004-volumes-on-normed-and-finsler-spaces}: a \emph{volume} is a function $\mu$ which assigns to every $n$-dimensional normed space $V$ a norm $\mu_V$ on the exterior power $\Lambda^n V$ in such a way that $\mu_V$ is induced by the Lebesgue measure if $V$ is Euclidean and such that for every linear $1$-Lipschitz map $T \colon V \to W$ between $n$-dimensional normed spaces, the induced map $T_* \colon \Lambda^n V \to \Lambda^n W$ is $1$-Lipschitz. Following \cite{Lyt:Wen:17:areamini}, we define the Jacobian $\mathbf{J}^{\mu}$ for a seminorm $s$ on $\mathbb{R}^n$ by
\begin{equation*}
    \mathbf{J}^{\mu}_{n}(s)
    =
    \left\{
    \begin{split}
        &\mu_{ ( \mathbb{R}^n, s ) }( e_1 \wedge \dots \wedge e_n ), \quad&&\text{if $s$ is a norm}
        \\
        &0, \quad&&\text{if $s$ is a seminorm,}
    \end{split}
    \right.
\end{equation*}
where $e_1, \dots, e_n$ is the standard basis in $\R^n$. The associated parametrized volume of a Sobolev map $u \in W^{1,p}( \Omega, X )$, for $p \geq n$, is
\begin{align*}
    \mathrm{Vol}^\mu( u ) \coloneqq \int_{\Omega} \mathbf{J}^{\mu}_n( \mathrm{md}_z(u) ) \,dz.
\end{align*}
When $n = 2$, we instead write $\mathrm{Area}^{\mu}(u)$ and $\mathbf{J}^{\mu}$.

We say that a volume functional $\mathrm{Vol}$ is \emph{lower semi-continuous} if for $p \in [n,\infty]$ and every complete metric space $X$, any bounded domain $\Omega \subset \mathbb{R}^n$, and any $p$-bounded sequence $( u_m \colon \Omega \to X )$ with $\lim_{ m \to \infty } \| d( u_m, u ) \|_{ L^{p}(\Omega) } = 0$, it holds that
\begin{align*}
    \liminf_{ m \to \infty } \mathrm{Vol}( u_m ) \geq \mathrm{Vol}( u ).
\end{align*}
Note that the lower semicontinuity for $p > n$ follows from the lower semicontinuity in the case $p = n$.

The parametrized volume $\mathrm{Vol}^{\mu}$ is lower semicontinuous if (and only if) $\mu$ induces quasi-convex $n$-volume densities; see \cite[Definition 2.6 and Corollary 5.8]{Lyt:Wen:17:areamini}. We call $\mu$ \emph{lower semi-continuous} if $\mathrm{Vol}^{\mu}$ is lower semi-continuous in the sense above. When $n = 2$, we instead use the term \emph{area}. In particular, the parametrized Hausdorff area is lower semi-continuous by \cite[Theorem 1]{burago-ivanov-2012-minimality-of-planes-in-normed-spaces}.

We give additional comments on the areas and refer to \cite[Section 2.4]{Lyt:Wen:17:energyarea} for further discussion. We will later need another notion of area, called the \emph{Riemann inscribed area} (due to Ivanov \cite{ivanov-2008-volumes-and-areas-lipschitz-metrics}): The $\mu^{i}$-Jacobian of a norm $s$ on $\mathbb{R}^2$ satisfies
\begin{align*}
    \mathbf{J}^{\mu^{i}}(s)
    =
    \frac{ \pi }{ \lm^{2}(L) }
\end{align*}
where $\lm^{2}(L)$ is the Lebesgue measure of the ellipse of maximal area contained inside the unit ball $\{s \leq 1\}$. That is, the Lebesgue area of the \emph{John ellipse} of $\{ s \leq 1 \}$. We note that
\begin{align*}
    \frac{\pi}{4} \mathbf{J}^{\mu^{i}}(s) \leq \mathbf{J}(s) \leq \mathbf{J}^{\mu^{i}}(s)
    \quad\text{for every seminorm in $s$.}
\end{align*}
In fact, the right-hand side is an equality if and only if $s$ is not a norm or it is induced by an inner product. 

In addition to $\mu^{i}$, other prominent lower semi-continuous areas include the Holmes--Thompson area $\mu^{ht}$, Benson's area $\mu^{*}$ (also known as Gromov's $\text{mass}^{*}$), and the Busemann--Hausdorff area $\mu^{bh}$. The latter one defines the same Jacobian as the parametrized Hausdorff area. The main inequalities are $\mu^{bh} \geq \mu^{ht}$ and $\mu^{*} \geq \mu^{ht}$ both of which are strict outside inner product spaces. We remark that the Jacobians of every notion of area $\mu$ satisfies $2^{-1}\mathbf{J}^{\mu^{i}} \leq \mathbf{J}^\mu \leq \mathbf{J}^{ \mu^{i} }$.

\section{Proof of \Cref{thm:embedding-mapping-packages}}\label{section-mapping-package}
We recall the statement of the embedding theorem. Below $\omega$ will denote a non-principal ultrafilter on $\N$.
\GromovEmbedding*

Before proceeding with the proof, we fix notation. We lose no generality in supposing that $I = \mathbb{N}$. Let $( \varepsilon_N )$ be a decreasing sequence of positive numbers with $\sum \varepsilon_N  < \infty$, and let $D^k = \{ y_{j}^{k} \colon j \in \mathbb{N} \} \subset Y^k$ be a countable dense set for $k \in \N$. Let also $X_\omega$ be the ultralimit of $( X_m, d_m, p_m )$ and let $f_\omega^{k} \colon Y^k \to X_\omega$ be the ultralimit of $( f_m^k \colon Y^k \to X_m )$ for $k \in \N$.

The construction of the space $Z$ is fairly standard in the context of related compactness theorems; see e.g. \cite[Proposition 5.2]{wenger-2011-compactness-for-manifolds-and-integral-currents-with-bounded-diameter-and-volume}. However, to guarantee compatibility with the ultralimit requires further care. To begin the proof, we will consider a decreasing family of sets $\mathcal{G}_N$ with $\omega( \mathcal{G}_N ) = 1$ for $N \geq 1$ with the following requirements: $m \in \mathcal{G}_N$ if and only if
\begin{enumerate}
    \item for $1 \leq j, j', k, k' \leq N$, it holds that
    \begin{align*}
        | d_m( f_{m}^{k}( y_{j}^{k} ), f_{m}^{k'}( y_{j'}^{k'} ) ) - d_\omega( f_\omega^{k}( y_{j}^{k} ), f_{\omega}^{k'}( y_{j'}^{k'} ) ) | < \varepsilon_{N}.
    \end{align*}
    \item $N \leq m$.
\end{enumerate}
We have $\mathcal{G}_{N+1} \subset \mathcal{G}_N$ for $N \in \N$ and $\bigcap \mathcal{G}_N = \emptyset$. Moreover, $\omega( \mathcal{G}_{N} ) = 1$ for $N \geq 1$ since $\mathcal{G}_N$ can be expressed as a finite intersection of full $\omega$-measure sets.

We now proceed with the construction of the sequence $( m_l )$. To this end, for each $m \in \mathcal{G}_1$, let $N = N(m)$ be the largest integer such that $m \in \mathcal{G}_N$. Next, let $m_1$ be an element in $\mathcal{G}_1$ and let $N_1 \coloneqq N(m_1)$. Suppose that $m_{l-1}$ and $N_{l-1}$ have been defined for some $l \geq 2$ and then consider any element $m_l \in \mathcal{G}_{ m_{l-1}+1 }$, and denote $N_l = N(m_l)$. It follows that $( N_l )$ and $( m_l )$ are strictly increasing and $N_l \leq m_l$.

We are ready to construct the space $Z$ via a direct limit construction. To this end, let $Z^1 = X_{m_1}$ and let $d^{1} \coloneqq d_{ m_1 }$. Suppose that $Z^{l-1}$ and $d^{l-1}$ have been defined for some $l \geq 2$. Next, consider the disjoint union $Z^{l} = Z^{l-1} \coprod X_{ m_{l} }$ and define a distance $d^{l}$ as follows. In $Z^{l-1}$, the distance restricts to $d^{l-1}$ and on $X_{m_l}$ the distance restricts to $d_{ m_l }$. For points $(x,x') \in Z^{l-1} \times X_{ m_l }$, we define
\begin{align*}
    d^{l}(x',x)
    \coloneqq
    d^{l}(x,x')
    \coloneqq
    \min\left\{ d^{l-1}(x, f_{m_{l-1}}^{k}( y^{k}_j ) ) + ( \varepsilon_{N_{l-1}} + \varepsilon_{ N_{l} } ) + d_{m_l}( f_{m_{l}}^{k}( y^{k}_j ), x' ) \right\},
\end{align*}
where the minimum is taken over $1 \leq j, k \leq N_{ l-1 }$. Regarding triangle inequality, by applying (1) for $m_{l}$ and $m_{l-1}$, respectively, we obtain that
\begin{align*}
    | d_{m_l}( f_{m_l}^{k}(y^{k}_j), f_{m_l}( y^{k'}_{j'} )) - d_{m_{l-1}}( f_{m_{l-1}}^{k}(y^{k}_j), f_{m_{l-1}}( y^{k'}_{j'} ) ) |
    <
    ( \varepsilon_{N_{l-1}} + \varepsilon_{ N_{l} } )
\end{align*}
when $1 \leq j,j',k, k' \leq N_{l-1}$. This and the recursive assumption on $d^{l-1}$ easily implies the triangle inequality for $d^{l}$. It follows that $d^l$ is an actual distance on $Z^l$ with the desired properties. Next, we define $Z$ as the completion of the direct limit of the spaces $( Z^{l}, d^{l} )$ and we denote the distance on $Z$ by $d$ for simplicity. We also suppress the notation for the isometric embeddings $X_{ m_l } \xhookrightarrow{} Z$.

Let $k \in \mathbb{N}$. We claim that for $y \in Y^k$, the sequence $( f_{m_l}^k(y) )$ is a Cauchy sequence, whose limit we denote by $f_\infty^{k}(y)$. This defines implicitly a map $f_\infty^k \colon Y^k \to Z$. We first consider the case $y = y_{j}^k$ for some $j \in \N$. To this end, let $l_0 \in \mathbb{N}$ such that $1 \leq j, k \leq N_{l_0-1}$. When $l \geq l_0$, we have that
\begin{align*}
    d( f_{ m_l }^{k}( y_{j}^{k} ), f_{ m_{l+1} }^{k}( y_{j}^{k} ) )
    =
    ( \varepsilon_{N_{l}} + \varepsilon_{ N_{l+1} } )
\end{align*}
so
\begin{align*}
    \sum_{ l = l_{1} }^{ l_{2} } d( f_{ m_l }^{k}( y_{j}^{k} ), f_{ m_{l+1} }^{k}( y_{j}^{k} ) )
    \leq
    2
    \sum_{ l = N_{ l_1 } }^{ \infty } \varepsilon_{ l }
    \quad\text{for $l_2 \geq l_1 \geq l_0$}.
\end{align*}
The right-hand side converges to zero as $l_1,l_2 \rightarrow \infty$ so $( f_{m_l}^k(y) )$ is Cauchy. The general case follows by equicontinuity and density of $D^{k} \subset Y^k$.

Next, given $j, j', k, k' \in \mathbb{N}$, property (1) in the construction guarantees that
\begin{align*}
    d( f_{\infty}^{k}( y_{j}^{k} ), f_{\infty}^{k'}( y_{j'}^{k'} ) )
    =
    \lim_{ l \rightarrow \infty }
    d( f_{m_l}^{k}( y_{j}^{k} ), f_{m_l}^{k'}( y_{j'}^{k'} ) )
    =
    d_\omega( f_{\omega}^{k}( y_{j}^k ), f_{\omega}^{k'}( y_{j'}^{k'} ) ).
\end{align*}
By equicontinuity, this easily extends to
\begin{align*}
    d( f_{\infty}^k(y), f_{\infty}^{k'}(y') ) = d_{\omega}( f_{\omega}^{k}(y), f_{\omega}^{k'}(y') )
    \quad\text{for $(y,y') \in Y^k \times Y^{k'}$ and $k,k' \in \N$.}
\end{align*}
We conclude that there exists a uniquely defined isometric embedding $\eta \colon C \to X_\omega$ such that $\eta \circ f_{\infty}^{k} = f_{\omega}^k$ for $k \geq 1$. For each $k \in \N$, the equicontinuity of $( f_m^k )$ implies that the pointwise convergence of $( f_{m_l}^{k} \colon Y^k \to X_{m_l} )$ to $f_{\infty}^k \colon Y^k \to Z$ improves to uniform convergence on compact sets. Thus (1) and (2) in the statement of the claim hold.

To finish, we address the flexibility part of the claim. If $( \mathcal{F}_N )_{ N \geq 1 }$ is as stated in the statement of the theorem and $( \mathcal{G}_N )_{ N \geq 1 }$ is as above, we may replace $\mathcal{G}_N$ by $\mathcal{G}_N' = \mathcal{G}_N \cap \mathcal{F}_N$ in the construction of the sequences $(N_l)$ and $( m_l )$ but otherwise the proof is the same. This completes the proof of \Cref{thm:embedding-mapping-packages}.

\section{Lusin--Lipschitz admissible sequences}\label{section-lusin-lipschitz-admissible-sequences}
In this section, we study quantitative Lusin--Lipschitz admissible sequences, defined below. They will be a crucial ingredient in the proofs of the ultralimit theorems. Let $X$ be a complete metric space and $\Omega \subset\R^n$ a bounded Lipschitz domain.

\begin{definition}\label{definition-lusin-lipschitz-admissible-sequence}
  A sequence of Lipschitz maps $f^k\colon \Omega\to X$ is called Lusin--Lipschitz admissible if there exist non-empty measurable subsets $E^1\subset E^2\subset \dots\subset\Omega$ such that $\lm^n(\Omega\setminus E^k)\to 0$ as $k$ tends to infinity and $f^{k+1}|_{E^k} = f^k|_{E^k}$ for every $k$.
\end{definition}

If $\lambda_k, c_k>0$ with $c_k\to 0$ are such that $f^k$ is $\lambda_k$-Lipschitz and $\lm^n(\Omega\setminus E^k)\leq c_k$ for every $k$, we call $\lambda_k$ and $c_k$ the parameters of the Lusin--Lipschitz admissible sequence $(f^k)$. Clearly, for almost every $z\in\Omega$, the sequence $(f^k(z))$ is eventually constant and, in particular, the pointwise limit $f(z)= \lim_{k\to\infty} f^k(z)$ is an almost everywhere defined measurable and essentially separably valued map from $\Omega$ to $X$. This observation will be at the basis of the ultralimit construction.

The following proposition shows that a Lusin--Lipschitz admissible sequence $( f^k )$ converges to the pointwise limit $f$ in $L^p$ quantitatively in the parameters $\lambda_k$ and $c_k$.

\begin{proposition}\label{prop:distance-admissible-sequence-limit}
   Given $C,p>1$ there exist $C',\delta>0$ with the following properties. For every Lusin--Lipschitz admissible sequence of maps $f^k\colon\Omega\to X$ with parameters $\lambda_k= C\cdot 2^k$ and $c_k= \lm^n(\Omega) \cdot 2^{-kp}$ and limit map $f$, the real-valued function $h_k= d(f,f^k)$ satisfies $$\|h_k\|_{L^p(\Omega)} \leq 2^{-k\delta}C'$$ for every $k\in\N$.
\end{proposition}

The proof will show that $\delta$ depends on $p$ and $n$ and $C'$ depends on $p$, $n$, $C$, and $\Omega$.

\begin{proof}
 Let $C, p>1$ and for $k\in\N$ set $\lambda_k= C\cdot 2^k$ and $c_k= \lm^{n}(\Omega)\cdot 2^{-kp}$. Suppose $(f^k)$ is a Lusin--Lipschitz admissible sequence of maps $f^k\colon\Omega\to X$ with parameters $\lambda_k$ and $c_k$. That is, $f^k$ is $\lambda_k$-Lipschitz and there exist non-empty measurable subsets $E^1\subset E^2\subset\dots \subset \Omega$ such that $\lm^n(\Omega\setminus E^k)\leq c_k$ and $f^{k+1}|_{E^k} = f^k|_{E^k}$ for every $k$. The limit $f$ of the sequence $(f^k)$ satisfies $f|_{E^k} = f^k|_{E^k}$ for every $k$.
 
 Fix $k\in\N$ and abbreviate $h=h_k$. We claim that $h\in W^{1,q}(\Omega)$ with Sobolev norm $$\|h\|_{W^{1,q}(\Omega)}\leq 2^{\left(1-\frac{p}{q}\right)k}D$$ for every $1<q<p$, where $D$ depends on $p$, $q$, $C$, $n$, and $\diam(\Omega)$.  For this, observe first that $h|_{E^k}=0$ and, for $i> k$, we have $$h|_{E^i} = d(f^i, f^k)|_{E^i}\leq (\lambda_i + \lambda_k)\cdot\diam(\Omega)\leq 2^{i+1}C\cdot\diam(\Omega).$$
From this, we deduce that 
\begin{equation*}
    \begin{split}
        \int_\Omega h^q(z)\,dz &= \sum_{i=k+1}^\infty \int_{E^i\setminus E^{i-1}}h^q(z)\,dz\\
        &\leq (2C\diam(\Omega))^q\cdot \sum_{i=k+1}^\infty 2^{qi}\cdot \lm^n(\Omega\setminus E^{i-1})\\
 &\leq 2^{(q-p)k} C_0
    \end{split}
\end{equation*}
for some constant $C_0$ depending only on $p, q, C$, $n$, and $\diam(\Omega)$. Now, for $l\geq k$, consider the Lipschitz function $g_l=d(f^l,f^k)$. Since $$|g_l-h_k| = |d(f^l, f^k) - d(f^k, f)| \leq h_l,$$ it follows from the above that $g_l$ converges to $h_k$ in $L^q$. Notice that $g_l|_{E^k}=0$ and $g_l$ is $2\lambda_i$-Lipschitz on $E^i$ for every $i\geq l$. In particular, $|\nabla g_l|=0$ almost everywhere on $E^k$ and $|\nabla g_l|\leq 2\lambda_i$ almost everywhere on $E^i$. Therefore, we obtain almost as above
\begin{equation*}
      \int_\Omega |\nabla g_l|^q\,dz \leq (2C)^q \sum_{i=k+1}^\infty 2^{qi}\lm^n(\Omega\setminus E^{i-1}) \leq 2^{(q-p)k}C_1
\end{equation*}
for some constant $C_1$ only depending on $p$, $q$, $n$, and $C$. Thus, the $L^q$-norm of $|\nabla g_l|$ is bounded independently of $l$ and since $g_l$ converges in $L^q$ to $h$, it follows that $h\in W^{1,q}(\Omega)$ and $$\int_\Omega |\nabla h|^q\,dz\leq 2^{(q-p)k}C_1.$$ This yields the claim.

    Finally, consider $1<q<\min\{p, n\}$ such that $\frac{qn}{n-q}>p$. By the Sobolev embedding theorem and H\"older's inequality, there exists a constant $D'$ only depending on $p, q, n$, and $\Omega$ such that $$\|h\|_{L^p(\Omega)} \leq D' \|h\|_{W^{1,q}(\Omega)}\leq 2^{\left(1-\frac{p}{q}\right)k}D'D.$$ This concludes the proof with $\delta = \frac{p}{q} - 1$ and $C'= D' D$.
\end{proof}

\begin{proposition}\label{prop:Sobolev-limit-admissible-sequence}
Given $p > 1$ and $M \geq 1$, there exists $C \geq 1$, depending only on $M$, $\Omega$ and $p$, with the following property. Let $X$ be an injective metric space, $p_X\in X$, and let $f\in W^{1,p}(\Omega, X)$ satisfy $$\int_\Omega d^p(f(z), p_X)\,dz+ E_+^p(f)\leq M.$$ Then $f$ is the limit of a Lusin--Lipschitz admissible sequence $(f^k)$ with parameters $\lambda_k = C \cdot 2^{k}$ and $c_k= \lm^n(\Omega) \cdot 2^{-kp}$ and moreover satisfying $f^k(\Omega)\subset \overline{B}_{ X }(p_X, \lambda_k)$ for every $k$.
\end{proposition}

The condition that $X$ be an injective metric space can be weakened. It is enough to assume that $X$ is Lipschitz $(n-1)$-connected; see e.g. \cite{lang-schlichenmaier-2005-nagata-dimension-quasisymmetric-embeddings-and-lipschitz-extensions} or \cite{hajlasz-schikorra-2014-lipschitz-homotopy-and-density-of-lipschitz-mappings-in-sobolev-spaces} for related extension constructions. In that case, the constant $C$ also depends on the data of the Lipschitz connectivity.

\begin{proof}
Let $p_X$ and $f$ be as in the statement of the proposition. Consider a Haj{\l}asz gradient $g$ of $f$ satisfying
\begin{align*}
    \| g \|_{ L^{p}( \Omega ) }^p \leq C' E^{p}_{+}(g).
\end{align*}
Recall that we may choose $C' \geq 1$ to depend only on $\Omega$ and $p$ by \eqref{eq:normestimate:hajlasz:sobolev}.

Consider the sublevel sets
\begin{align*}
    A_k &= \left\{ g \leq 2^{k} \right\}
    \quad\text{and}\quad
    B_k = \left\{ d( f, p_X ) \leq 2^{k} \right\}
\end{align*}
for $k\geq 1$. We deduce from Chebyshev inequality that
\begin{align*}
    \lm^{n}( \Omega \setminus A_k ) &\leq C'\frac{ E^{p}_+(u) }{ 2^{kp} }
    \quad\text{and}\quad
    \lm^n( \Omega \setminus B_k ) \leq \frac{ \| d(f,p_X) \|_{ L^{p}( \Omega ) }^p }{ 2^{kp} }.
\end{align*}
Let $k_0 \in \N$ be so large that $C' M / 2^{k_0 p } < \lm^{n}(\Omega)$. Then, for $k\geq 1$, the set $E^k = A_{k+k_0} \cap B_{k+k_0}$ satisfies
\begin{align*}
    \lm^{n}( \Omega \setminus E^k )
    \leq
    \frac{ C'M }{ 2^{k_0 p} } \frac{ 1 }{ 2^{kp} }
    <
    \frac{ \lm^n( \Omega ) }{ 2^{kp} }
\end{align*}
and, in particular, $E^k$ is non-empty. Moreover, by \eqref{eq:hajlaszgradient}, the restriction $f|_{ E^{k} }$ is $C 2^{k}$-Lipschitz with $C = 2^{k_0+1}$, and $f( E^k ) \subset \overline{B}_X( p_X, C 2^{k} )$.

Now as $X$ is injective, we may apply the injectivity of $\overline{B}_{X}( p_X, C 2^{k} )$ to find a $\lambda_k$-Lipschitz extension $f^k$ of $f|_{E^{k}}$ satisfying $f^{k}( \Omega ) \subset \overline{B}_{X}( p_X, C 2^{k} )$. The sequence $( f^k )$ has the desired properties.
\end{proof}

\section{Proof of \Cref{theorem-ultralimit-construction,theorem-ultralimit-trace-properties}}\label{section-the-ultralimit-construction}
Below $\omega$ will denote a non-principal ultrafilter on $\N$. For each $m\in\N$, let $X_m=(X_m, d_m, p_m)$ be a pointed, complete metric space and consider a $p$-bounded sequence $(u_m \colon \Omega \to X_m)$. That is, $u_m\in W^{1,p}(\Omega, X_m)$ and $$\sup_{m\in\N}\left(\int_\Omega d^p(u_m(z), p_m)\,dz+ E_+^p(u_m)\right)<\infty.$$
Denote by $X'_m$ the injective hull of $X_m$; see \Cref{section-injective-metric-spaces}. We view $X_m$ as a subset of $X'_m$ and $u_m$ as an element of $W^{1,p}(\Omega, X'_m)$. We denote the metric on $X'_m$ by $d_m$ again. Denote by $X_\omega$ and $X'_\omega$ the ultralimits of the sequences $(X_m, d_m, p_m)$ and $(X'_m, d_m, p_m)$, respectively. 

By \Cref{prop:Sobolev-limit-admissible-sequence}, there exists $C$ such that each map $u_m$ is the limit of a Lusin--Lipschitz admissible sequence $(u_m^k)$ of maps $u_m^k\colon\Omega\to X'_m$ with parameters $\lambda_k=C2^k$ and $c_k=\lm^{n}(\Omega)2^{-pk}$ and moreover $$u_m^k(\Omega)\subset \overline{B}_{X'_m}(p_m, \lambda_k).$$ For each $k\in\N$, denote by $u_\omega^k$ the pointwise ultralimit of the sequence $(u_m^k)$.

\begin{lemma}\label{lemma-image-is-in-the-correct-place}
    The sequence $( u_\omega^k )$ is Lusin--Lipschitz admissible with parameters $\lambda_k$ and $c_k$. Moreover, after possibly redefining the limit $u_\omega$ of $( u_\omega^k )$ on a set of measure zero, $u_\omega$ has image in $X_\omega$.
\end{lemma}

\begin{proof}
    Notice first that $u_\omega^k$ is $\lambda_k$-Lipschitz for every $k$. Now, for each $m$, let $$E_m^1\subset E_m^2\subset\dots\subset \Omega$$ be non-empty measurable sets such that $\lm^n(\Omega\setminus E_m^k)\leq c_k$ and $$u_m^{k+1}|_{E_m^k} = u_m^k|_{E_m^k}$$ for every $k$. Let $E_\omega^k$ be the ultralimit of $( E_m^k )$; see \Cref{section-ultrafilters}. Set $F^k=E_\omega^k\cap \Omega$. We clearly have $$F^1\subset F^2\subset\dots\subset\Omega$$ and, by \Cref{lem:bound-ultralimit-set}, we have $\lm^n(\Omega \setminus F^k)\leq c_k$ so every $F^k$ is non-empty. Next, letting $z\in F^k$, it holds that $z=\lim\nolimits_{\omega} z_m$ for some $z_m\in E_m^k$. Since 
    \begin{equation*}
    \begin{split}
        d_m(u_m^{k+1}(z), u_m^k(z))&\leq d_m(u_m^{k+1}(z), u_m^{k+1}(z_m)) + d_m(u_m^k(z_m), u_m^k(z))\\
        &\leq (\lambda_{k+1}+\lambda_k)\cdot |z-z_m|,
    \end{split}
    \end{equation*}
    it follows that $u_\omega^{k+1}(z) = u_\omega^k(z)$ and hence $$u_\omega^{k+1}|_{F^k} = u_\omega^k|_{F^k}.$$ We conclude that $(u_\omega^k)$ is a Lusin--Lipschitz admissible sequence with parameters $\lambda_k$ and $c_k$. Similarly, noting that $u_m^k(z_m) = u_m(z_m)$ and $$d_m(u_m^k(z), u_m(z_m))\leq \lambda_k |z-z_m|,$$ it follows that
    $$u_\omega(z) = \lim\nolimits_\omega u_m^k(z) = \lim\nolimits_\omega u_m(z_m)\in X_\omega.$$
    We conclude that $u_\omega$ maps into $X_\omega$ after possibly redefining $u_\omega$ on the set $\Omega \setminus \bigcup F^k$ of measure zero.
\end{proof}

We call $u_\omega$ the ultralimit associated with the Lusin--Lipschitz admissible sequences $(u_m^k)$. From now on, we will view $u_\omega$ as a map from $\Omega$ to $X_\omega$.

\begin{lemma}\label{lemma-lusin-lipschitz-admissibility}
    The ultralimit $u_\omega$ associated with $( u_m^k )$ belongs to $W^{1,p}(\Omega, X_\omega)$. Moreover,
    \begin{equation*}
        \lim\nolimits_{\omega} E^{p}( u_m ) \geq E^{p}( u_\omega )
    \end{equation*}
    for any lower semi-continuous energy and, if $p \geq n$,
    \begin{equation*}
        \lim\nolimits_{\omega} \mathrm{Vol}( u_m ) \geq \mathrm{Vol}( u_\omega )
    \end{equation*}
    for any lower semi-continuous volume.
\end{lemma}

\begin{proof}
By \Cref{thm:embedding-mapping-packages}, there exist a complete metric space $Z$, a subsequence $(m_l)$ and isometric embeddings $\iota_l\colon X'_{m_l}\hookrightarrow Z$ such that, for each $k$, the sequence of maps $f_{m_l}^{k} = \iota_l\circ u_{m_l}^k$ converges uniformly to a map $f_\infty^k\colon \Omega\to Z$. Moreover, denoting by $C$ the closure of $\bigcup f_\infty^k(\Omega)$, there exists an isometric embedding $\eta\colon C\hookrightarrow X_\omega'$ such that $u_\omega^k = \eta\circ f_\infty^k$ for every $k$. This factorization guarantees that the sequence $(f^k_\infty)$ is Lusin--Lipschitz admissible with parameters $\lambda_k$ and $c_k$ by \Cref{lemma-image-is-in-the-correct-place}. We claim that its limit $f_\infty$ belongs to $W^{1,p}(\Omega, Z)$. Since the maps $f_{m_l} = \iota_l\circ u_{m_l}$ have uniformly bounded energy, it suffices to show that they converge to $f_\infty$ in $L^p$.

Let $\varepsilon>0$. By \Cref{prop:distance-admissible-sequence-limit}, there exists $k\in\N$ such that $$\|d(f_\infty, f_\infty^k)\|_{L^p(\Omega)}\leq \frac{\varepsilon}{3}\quad\text{ and }\quad\|d_m(u_m, u_m^k)\|_{L^p(\Omega)}\leq \frac{\varepsilon}{3}$$ for every $m$. Since $f_{m_l}^k$ converges uniformly to $f_\infty^k$ as $l\to\infty$, we finally obtain that $$\|d(f_\infty, f_{m_l})\|_{L^p(\Omega)}\leq \frac{\varepsilon}{3} + \|d(f_\infty^k, f_{m_l}^{k})\|_{L^p(\Omega)} + \frac{\varepsilon}{3}\leq \varepsilon$$ for all sufficiently large $l$. As $\varepsilon$ was arbitrary, this shows that $f_{m_l}$ converges to $f_\infty$ in $L^p$ and hence $f_\infty\in W^{1,p}(\Omega, Z)$.
Since $u_\omega^k = \eta \circ f_\infty^k$ for every $k$, it holds that $u_\omega = \eta \circ f_\infty$ almost everywhere so $u_\omega$ belongs to $W^{1,p}(\Omega, X_\omega)$.

Given a lower semi-continuous notion of an energy and volume (if $p \geq n$), when choosing the sequence $(m_l)$, we may also suppose that
\begin{align*}
    \lim_{l \to\infty} \mathrm{Vol}( u_{m_l} ) = \lim\nolimits_\omega \mathrm{Vol}( u_m ) \quad\text{if $p \geq n$}
    \quad\text{and}\quad
    \lim_{l\to\infty} E^{p}( u_{m_l} ) = \lim\nolimits_\omega E^{p}( u_{m} ).
\end{align*}
The claim follows.
\end{proof}

The next lemma implies that the map $u_\omega$ does not depend on the choice of the Lusin--Lipschitz admissible sequences $(u_m^k)$.

\begin{lemma}\label{lemma-control-on-the-Lp-norms}
    If $(v_m \colon \Omega \to X_m )$ is another $p$-bounded sequence with associated Lusin--Lipschitz admissible sequences $(v_m^k \colon \Omega \to X_m' )$ and ultralimit $v_\omega$, then $$\|d_\omega(u_\omega, v_\omega)\|_{L^p(\Omega)} = \lim\nolimits_\omega \|d_m(u_m, v_m)\|_{L^p(\Omega)}$$ and $$\|d_\omega(\trace(u_\omega),\trace(v_\omega))\|_{L^{p}(\partial \Omega)} = \lim\nolimits_\omega \|d_m( \trace(u_m), \trace(v_m) )\|_{L^{p}(\partial \Omega)}.$$ 
\end{lemma}

This immediately implies the following key properties.
\begin{corollary}
    If $(u_m)$ is Lipschitz bounded, then $u_\omega$ agrees (almost everywhere) with the pointwise ultralimit. Moreover, if $(\varphi_m)$ is a Lipschitz bounded sequence of maps $\varphi_m\colon X_m\to Y_m$, then the sequence of maps $v_m = \varphi_m\circ u_m$ satisfies $v_\omega=\varphi_\omega\circ u_\omega$ almost everywhere.
\end{corollary}
The first part is immediate from \Cref{lemma-control-on-the-Lp-norms}. The second part uses the additional fact that $\varphi_m \colon X_m \to Y_m$ has an extension $\Phi_m \colon X_{m}' \to Y_m'$ to the injective hulls so that $( \Phi_m )$ is Lipschitz bounded.

\begin{proof}[Proof of \Cref{lemma-control-on-the-Lp-norms}]
    The proof uses similar ideas as the one of \Cref{lemma-lusin-lipschitz-admissibility}. Indeed, by \Cref{thm:embedding-mapping-packages}, there exist a complete metric space $Z$, a subsequence $(m_l)$ and isometric embeddings $\iota_l\colon X'_{m_l}\hookrightarrow Z$ such that, for each $k$, the sequence of maps $f_{m_l}^{k} = \iota_l\circ u_{m_l}^k$ and $g_{m_l}^k = \iota_l \circ v_{m_l}^k$ converge uniformly to maps $f_\infty^k\colon \Omega\to Z$ and $g_\infty^k \colon \Omega \to Z$, respectively. Moreover, denoting by $C$ the closure of $\bigcup (f_\infty^k(\Omega) \cup g_\infty^k(\Omega))$, there exists an isometric embedding $\eta\colon C\hookrightarrow X_\omega'$ such that $u_\omega^k = \eta\circ f_\infty^k$ and $v_\omega^k = \eta \circ g_\infty^k$ for every $k$. In the choice of $( m_l )$, we may also suppose that
    \begin{align*}
        \lim_{l \to\infty} \|d_{m_l}(u_{m_l}, v_{m_l})\|_{L^p(\Omega)} = \lim\nolimits_\omega \|d_m(u_m, v_m)\|_{L^p(\Omega)}
    \end{align*}
    and
    \begin{align*}
        \lim_{l \to\infty} \|d_{m_l}(\trace(u_{m_l}), \trace(v_{m_l}))\|_{L^p(\partial \Omega)} = \lim\nolimits_\omega \|d_m(\trace(u_m), \trace(v_m))\|_{L^p(\partial \Omega)}.
    \end{align*}
    Let $f_\infty$ and $g_\infty$ denote the pointwise limits of $( f_\infty^k )$ and $( g_\infty^k )$ as in the proof of \Cref{lemma-lusin-lipschitz-admissibility}. We recall from the proof that $u_\omega = \eta \circ f_\infty$ and $v_\omega = \eta \circ g_\infty$ almost everywhere, and that $f_{m_l}$ and $g_{m_l}$ converge to $f_\infty$ and $g_\infty$ in $L^p$. It easily follows that
    \begin{align*}
        \lim_{l \to\infty} \|d_{m_l}(u_{m_l}, v_{m_l})\|_{L^p(\Omega)}
        &=
        \lim_{l \to\infty} \|d(f_{m_l}, f_{m_l})\|_{L^p(\Omega)}
        \\
        &=
        \| d(f_\infty,g_\infty) \|_{L^{p}(\Omega)}
        =
        \| d_\omega( u_\omega, v_\omega ) \|_{L^{p}(\Omega)}.
    \end{align*}
    The continuity properties of the trace imply the claim about traces; see \Cref{section-trace-theory}.
\end{proof}

We have proved that the assignment $(u_m)\mapsto u_\omega$ satisfies the properties stated in \Cref{theorem-ultralimit-construction}. To finish the proof of the theorem, it remains to address uniqueness.
\begin{lemma}
If an assignment $(u_m) \mapsto \widetilde{u}_\omega$ on $p$-bounded sequences satisfies \eqref{ultralimit:lipschitz-compatibility}, \eqref{ultralimit:postcomposition-compatibility}, and \eqref{ultralimit:Lp-compatibility}, then $u_\omega = \widetilde{u}_\omega$ almost everywhere for every $p$-bounded $(u_m)$.
\end{lemma}
\begin{proof}
Given $( u_m^k )$ as above, the sequence $(u_m^k)$ is Lipschitz bounded for $k \in \N$, and \Cref{prop:distance-admissible-sequence-limit} guarantees $$\|d_m(u_m, u_m^k)\|_{L^p(\Omega)}\leq 2^{-k\delta}C' \quad\text{for $k \in \N$,}$$ where $\delta$ and $C'$ are independent of $k$ and $m$. Therefore, by the axioms, it holds that
\begin{align*}
    \| d_\omega( \tilde{u}_\omega, u_\omega^k ) \|_{ L^{p}(\Omega) }
    =
    \lim\nolimits_\omega 
    \|d_m(u_m, u_m^k)\|_{L^p(\Omega)}\leq 2^{-k\delta}C'
\end{align*}
and similarly that $\| d_\omega( u_\omega, u_\omega^k ) \|_{ L^{p}(\Omega) } \leq 2^{-k\delta}C'$. The uniqueness follows by passing to the limit as $k \to \infty$.
\end{proof}

To finish the proof of \Cref{theorem-ultralimit-trace-properties}, it remains to verify \eqref{ultralimit:compatibility-of-traces}.
\begin{lemma}
If the traces $\trace(u_m) \colon \partial \Omega \to X_m$ have continuous representatives $( \gamma_m )$ from an equi-continuous and equi-bounded family, then the pointwise ultralimit $\gamma_\omega$ of $( \gamma_m )$ is a continuous representative of $\trace( u_\omega )$.
\end{lemma}
\begin{proof}
    By \Cref{thm:embedding-mapping-packages}, there exist a complete metric space $Z$, a subsequence $(m_l)$ and isometric embeddings $\iota_l\colon X'_{m_l}\hookrightarrow Z$ such that, for each $k$, the sequence of maps $f_{m_l}^{k} = \iota_l\circ u_{m_l}^k$ converges uniformly to a map $f_\infty^k\colon \Omega\to Z$ and $\iota_l \circ \gamma_{m_l}$ converge uniformly to a map $\gamma_\infty$. Moreover, denoting by $C$ the closure of $\bigcup ( f_\infty^k(\Omega) \cup \gamma_{\infty}(\partial \Omega) )$, there exists an isometric embedding $\eta\colon C\hookrightarrow X_\omega'$ such that $u_\omega^k = \eta\circ f_\infty^k$ for every $k$ and $\gamma_\omega = \eta \circ \gamma_\infty$.

    We use the notation and arguments from the proof of \Cref{lemma-lusin-lipschitz-admissibility}. Here $\trace( f_{m_l} ) = \iota_l \circ \trace( u_{m_l} ) = \iota_l \circ \gamma_{m_l}$ almost everywhere. Given that the traces of $( f_{m_l} )$ converge to the trace of $f_\infty$ in $L^{p}( \partial \Omega, Z )$ by the $L^p$-convergence of $( d( f_{m_l}, f_\infty ) )$ to zero, it follows that $\trace( f_\infty ) = \gamma_\infty$ almost everywhere. Recalling that $u_\omega = \eta \circ f_\infty$ almost everywhere, it holds that $\trace(u_\omega) = \eta \circ \trace( f_\infty ) = \eta \circ \gamma_\infty = \gamma_\omega$ almost everywhere.
\end{proof}
The proofs of \Cref{theorem-ultralimit-construction,theorem-ultralimit-trace-properties} are now complete.

\section{Stability results}\label{section-stability-results}
We apply the ultralimit theorems, \Cref{theorem-ultralimit-construction,theorem-ultralimit-trace-properties}, in the coming subsection in the proof of \Cref{thm:stability-introduction}. In the subsequent subsection, we prove the stability of Dehn functions, \Cref{theorem-stability-dehn-function}, as an application of \Cref{thm:stability-introduction}.

We will actually prove a version of these results for any lower semicontinuous area $\mu$, so we need the following definition: given a Lipschitz curve $\gamma \colon \mathbb{S}^1 \rightarrow X$ in a complete metric space $X$, the \emph{Sobolev filling area with respect to $\mu$} of $\gamma$ is
\begin{align*}
    \mathrm{FillArea}^{\mu}_{X}(\gamma) = \inf \left\{ \mathrm{Area}^{\mu}(u) \colon u \in W^{1,2}( \mathbb{D}, X ), \, \mathrm{tr}(u) = \gamma \right\}.
\end{align*}
When $\mu$ is omitted, we refer to the parametrized Hausdorff area of $\gamma$.

Below $\omega$ will denote a non-principal ultrafilter on $\N$.

\subsection{Proof of \Cref{thm:stability-introduction}}\label{section-lower-semicontinuity-of-areas}

We formulate the following general version of \Cref{thm:stability-introduction}.
\begin{theorem}\label{thm:stability}
Let $\mu$ be a lower semi-continuous area. Suppose that $( X_m, d_m, p_m )$ is a sequence of complete pointed length spaces and $X_\omega$ the ultralimit. If $( \gamma_m \colon \mathbb{S}^1 \to X_m )$ is a Lipschitz bounded sequence with $\gamma = \lim\nolimits_{\omega} \gamma_m$ and $\lim\nolimits_{\omega} \mathrm{FillArea}^{\mu}_{X_m}( \gamma_m ) < \infty$, then there exists $u \in W^{1,2}( \mathbb{D}, X_\omega )$ with
\begin{align*}
    \trace(u) = \gamma
    \quad\text{and}\quad
    \mathrm{Area}^{\mu}( u )
    \leq
    \lim\nolimits_\omega \mathrm{FillArea}^{\mu}_{X_m}( \gamma_m ).
\end{align*}
\end{theorem}

\begin{proof}
Let $L > \lim\nolimits_{\omega} \LIP( \gamma_m )$. We consider the subset $\mathcal{G} \subset \mathbb{N}$ defined by the properties $d( p_m, \gamma_m( \mathbb{S}^1 ) ) \leq 1 + d( p_\omega, \gamma( \mathbb{S}^1 ) )$, $\mathrm{FillArea}^{\mu}( \gamma_m ) \leq 1 + \lim\nolimits_{\omega} \mathrm{FillArea}^{\mu}( \gamma_m )$, and $\LIP( \gamma_m ) \leq L$. It holds that $\omega( \mathcal{G} ) = 1$. For $m \in \N \setminus \mathcal{G}$, we may assume that $X_m$ is a singleton.

Consider $u_m \in W^{1,2}( \mathbb{D}, X_m )$ such that $\mathrm{tr}(u_m) = \gamma_m$, and
\begin{align}\label{equation-filling-area-almost-optimal}
    \mathrm{Area}^{\mu}( u_m ) \leq \mathrm{FillArea}^{\mu}_{X_m}( \gamma_m ) + \frac{ 1 }{ m }.
\end{align}
We lift the sequence $( u_m )$ to the \emph{mapping cylinders} $Y_m$ of $\gamma_m$ next. To this end, equip the cylinder $Z = [0,3L] \times \mathbb{S}^{1}(L) \subset \mathbb{R}^3$ with the intrinsic length distance. Let $Y_m$ be the quotient space obtained by gluing $Z$ to $X_m$ such that $\gamma_m(s)$ is identified with $( 0, Ls )$ for each $s \in \mathbb{S}^{1}$ and equip $Y_m$ with the quotient metric.

Observe that by the choice of $L$, the natural inclusion $X_m \xhookrightarrow{} Y_m$ is isometric and the natural map $\iota_m \colon Z \rightarrow Y_m$ is $1$-Lipschitz and its restriction to $(0,3L] \times \mathbb{S}^{1}(L)$ is injective and a local isometry; in fact, the restriction of $\iota_m$ to the $L$-neighbourhood $N = [2L,3L] \times \mathbb{S}^{1}(L)$ of the top $T = \{3L\} \times \mathbb{S}^{1}(L)$ is an isometric embedding. Also, the projection map $P_m \colon Y_m \to X_m$ fixing $X_m$ pointwise and sending $(s,t)$ to $\gamma_m(L^{-1}t)$ for $(s,t) \in Z$ is $1$-Lipschitz. 

By the above, we may consider $u_m$ as an element of $W^{1,2}( \mathbb{D}, Y_m )$. By standard gluing properties of Sobolev maps \cite[Section 2.2]{Lyt:Wen:16}, we obtain a $W^{1,2}( \mathbb{D}, Y_m )$-mapping
\begin{equation*}
    v_m(z)
    =
    \left\{
    \begin{split}
        &u_m(2z), \quad&&\text{if $|z| < 2^{-1},$}
        \\
        &\left( 3L( 2|z|-1 ), L \frac{z}{|z|} \right), \quad&&\text{if $2^{-1} \leq |z| < 1$.}
    \end{split}
    \right.
\end{equation*}
It holds that
\begin{align*}
    \mathrm{Area}^{\mu}( v_m )
    &=
    \mathrm{Area}^{\mu}( u_m )
    +
    6\pi  L^2.
\end{align*}
By the metric Morrey's $\varepsilon$-conformality lemma \cite{Fitz:Wen:20}, there exists a diffeomorphism up to the boundary $\phi_m \colon \mathbb{D} \rightarrow \mathbb{D}$ such that $h_m = v_m \circ \phi_m$ satisfies
\begin{align*}
    E_{+}^{2}( h_m )
    \leq
    \frac{ 1 }{ m }
    +
    C_\mu \mathrm{Area}^{\mu}( v_m ),
\end{align*}
where $C_\mu$ is a constant depending only on $\mu$. In particular, after identifying the top $\Gamma_m \coloneqq \iota_m( T ) \subset Y_m$ isometrically with $T$ through $\iota_m$, we may suppose that the 3-point condition 
\begin{align*}
    \trace(h_m)( p_j ) = ( 3L, Lp_j )
    \quad\text{for $j \in \{1,2,3\}$ and $m \in \N$}
\end{align*}
holds for a fixed triple $\{p_1, p_2, p_3\}$ from $\mathbb{S}^1$. In particular, the traces $( \mathrm{tr}(h_m) )$ form an equi-continuous family by the Courant--Lebesgue lemma, cf. \cite[Lemma 7.3]{Lyt:Wen:17:areamini}.

It holds that $( h_m )$ is $2$-bounded by Poincaré inequality (see \cite[Lemma 4.11]{Lyt:Wen:17:areamini}) and that
\begin{equation}\label{equation-area-bounds-after-projections}
    \mathrm{Area}^{\mu}( P_m \circ h_m )
    =
    \mathrm{Area}^{\mu}( P_m \circ v_m )
    =
    \mathrm{Area}^{\mu}( u_m ).
\end{equation}
The second equality follows from the area formula \eqref{eq:areaformula:Sobolev}.

Next, we apply the ultralimit theorem, \Cref{theorem-ultralimit-construction}, yielding an ultralimit $h_\omega \in W^{1,2}( \mathbb{D}, Y_\omega )$ of $( h_m )$. Consider the ultralimit $P_\omega \colon Y_\omega \to X_\omega$ of the projections $( P_m )$. It holds that $P_\omega \circ h_\omega$ is an ultralimit of $( P_m \circ h_m )$ by \Cref{theorem-ultralimit-construction}. Moreover, using \Cref{theorem-ultralimit-trace-properties}, we have the following properties: First,
\begin{align}\label{equation-area-bound-preliminary}
    \mathrm{Area}^{\mu}( P_\omega \circ h_\omega )
    \leq
    \lim\nolimits_{\omega}
    \mathrm{Area}^{\mu}( P_m \circ h_m )
    \leq
    \lim\nolimits_{\omega} \mathrm{FillArea}^{\mu}_{X_m}( \gamma_m )
\end{align}
by \eqref{equation-area-bounds-after-projections} and \eqref{equation-filling-area-almost-optimal}. Second, since the traces $( \mathrm{tr}(h_m) )$ form an equi-continuous sequence, the ultralimit $\trace( h_\omega )$ of the traces is a weakly monotone parametrization of $\Gamma_\omega = \iota_\omega( T )$; here $\iota_\omega \colon Z \rightarrow Y_\omega$ is the pointwise ultralimit of $( \iota_m )$.

While $P_\omega \circ h_\omega$ satisfies the correct area bound \eqref{equation-area-bound-preliminary}, its trace is not necessarily $\gamma$. Thus we need to modify $h_\omega$ before projecting down. We will find $v \in W^{1,2}( \mathbb{D}, Y_\omega )$ such that $\trace(v)(z) = ( 3L, L z )$ for $z \in \mathbb{S}^1$ and
\begin{align}\label{equation-initial-ultralimit-bounds-on-area}
    \mathrm{Area}^{\mu}( P_\omega \circ v ) \leq \mathrm{Area}^{\mu}( P_\omega \circ h_\omega ).
\end{align}
Clearly $u \coloneqq P_\omega \circ v$ satisfies the conclusions of the claim by \eqref{equation-initial-ultralimit-bounds-on-area} and \eqref{equation-area-bound-preliminary}.

A simple homotopy argument, cf. \cite[Lemma 2.6]{Lyt:Wen:16}, implies the existence of $v$ if we find $g \in W^{1,2}( \mathbb{D}, Y_\omega )$ satisfying \eqref{equation-initial-ultralimit-bounds-on-area} (with $v$ replaced by $g$) so that $\trace(g)$ is an orientation-preserving reparametrization of $\trace(h_\omega)$ and bi-Lipschitz. The map $g$ is obtained by modifying $h_\omega$ in piecewise smooth Jordan domains $\Omega_1, \Omega_2, \dots, \Omega_{2k}$ formed by a finite number of short circular arcs from circles centered at $\mathbb{S}^1$. The construction is such that $\partial \Omega^{-}_j = \mathbb{S}^1 \cap \partial \Omega$ are connected and cover $\mathbb{S}^1$. By considering suitable circles, we may require by \cite[Lemma 6.4]{stadler-wenger-2025-isoperimetric-inequalities-vs-upper-curvature-bounds} that each trace $\theta_j = \trace( h_\omega|_{\Omega_j} )$ is continuous and of arbitrarily small length; for us, $\ell( \theta_j ) < L$ suffices. To obtain $g|_{\Omega_j}$, we reparametrize $\theta_j$ on $\partial\Omega_j^{-}$ to have constant-speed and fill the resulting curve $\eta_j \colon \partial \Omega_j \to Y_\omega$ using the local isoperimetric inequality on $\iota_\omega(N)$. Here \eqref{equation-initial-ultralimit-bounds-on-area} holds by the area formula \eqref{eq:areaformula:Sobolev} because $\iota_\omega(Z)$ projects to $\gamma( \mathbb{S}^1 )$ using $P_\omega$.
\end{proof}

\subsection{Dehn functions in ultralimits}\label{section-dehn-function-in-ultralimits}
We apply \Cref{thm:stability} now. \Cref{theorem-stability-dehn-function} is a special case of the following general version.
\begin{theorem}\label{theorem-stability-dehn-function-coarse}
Let $r_0$ and $\delta$ be as in \Cref{theorem-stability-dehn-function}. Let $( r_m )$ be a sequence of positive numbers converging to zero and let $( \delta_m \colon (0,r_0) \to (0,\infty) )$ converge to $\delta$ pointwise in $(0,r_0)$. Suppose that $\mu$ is a lower semi-continuous area. If $( X_m, d_m, p_m )$ is a sequence of complete pointed length spaces with
\begin{align*}
    \delta_{X_m}^{\mu}(r) \leq \delta_m(r)
    \quad\text{for every $r \in (r_m,r_0)$},
\end{align*}
then the ultralimit $X_\omega$ satisfies
\begin{align*}
    \delta_{X_\omega}^{\mu}( r ) \leq \delta(r)
    \quad\text{for every $r \in (0,r_0)$.}
\end{align*}
\end{theorem}
\begin{proof}
Let $r\in(0,r_0)$ and let $\gamma \colon \mathbb{S}^1\rightarrow X_\omega$ be a Lipschitz curve with $0<\ell(\gamma)\leq r$. Then $\gamma$ is the ultralimit of a Lipschitz bounded sequence of curves $ \gamma_m \colon \mathbb{S}^1 \rightarrow X_m $ satisfying $\lim\nolimits_{\omega} \ell( \gamma_m ) = \ell( \gamma )$, cf. \cite[Corollary 2.6]{Lyt:Wen:You:20}. Let $r<s<r_0$ and note that $r_m<\ell(\gamma_m)<s$ for every $m \in \mathcal{G} \subset \N$ with $\omega( \N \setminus \mathcal{G} ) = 0$. Since $$\mathrm{FillArea}^{\mu}_{X_m}( \gamma_m ) \leq \delta_{X_m}^\mu(s) \leq \delta_m(s)$$ and since $\delta_m(s)\to\delta(s)$, we obtain $\mathrm{FillArea}^{\mu}_{X_\omega}(\gamma) \leq \delta(s)$ from \Cref{thm:stability}. This implies $\delta_{X_\omega}(r) \leq \delta(s)$. By right-continuity of $\delta$, the claim follows.
\end{proof}

\section{Applications}\label{section-applications}

In this section, we prove the applications \Cref{theorem-gromov's-theorem-introduction} and \Cref{theorem-local-cat(kappa)}. As before, $\omega$ will denote a non-principal ultrafilter on $\N$.

\subsection{Gromov hyperbolicity}\label{section-gromov-hyperbolicity}
We formulate and prove the following general version of \Cref{theorem-gromov's-theorem-introduction}. Recall that a complete geodesic metric space $X$ is Gromov hyperbolic if and only if all its asymptotic cones are metric trees, cf. \cite[Proposition 3.1.1]{drutu-2002-quasi-isometry-invariants-and-asymptotic-cones}.

\begin{theorem}\label{theorem-gromov's-theorem}
Suppose that $\mu$ is a lower semi-continuous area and $X$ is a complete geodesic space with
\begin{align*}
    \limsup_{ r \rightarrow \infty } \frac{ \delta_{X}^{\mu}(r) }{ r^2 } < C_\mu.
\end{align*}
Then every asymptotic cone $X_\omega$ of $X$ is a metric tree. That is, $X$ is Gromov hyperbolic.
\end{theorem}
The constant $C_\mu$ is defined by
\begin{align}\label{equation-the-isoperimetric-constant}
    C_{\mu} \coloneqq \inf\{ \mu_V( \mathbb{I}_V ) / \ell^{2}( \partial \mathbb{I}_V ) \},
\end{align}
where the infimum is taken over two-dimensional normed planes $V$ where $\mathbb{I}_V \subset V$ is an \emph{isoperimetric set}. Namely, the set whose $\mu$ area is the largest among sets of a fixed perimeter. For instance, for the Holmes--Thompson, Gromov's $\text{mass}^{*}$, the Busemann--Hausdorff, and Ivanov's inscribed Riemannian area, it holds that $C_\mu = 1/(4\pi)$ \cite{alvarez-thompson-2004-volumes-on-normed-and-finsler-spaces}. Thus \Cref{theorem-gromov's-theorem} implies \Cref{theorem-gromov's-theorem-introduction} as claimed.

For the proof of \Cref{theorem-gromov's-theorem}, we need a lemma.
\begin{lemma}\label{lemma-degenerate-area}
Suppose that $\mu$ is a lower semi-continuous area and $X$ is a complete length space with
\begin{align*}
    \limsup_{ r \rightarrow 0 } \frac{ \delta_{X}^{\mu}(r) }{ r^2 } < C_\mu.
\end{align*}
Then $\mathrm{Area}^{\mu}(u) = 0$ for every $u \in W^{1,2}( \mathbb{D}, X )$.
\end{lemma}
\begin{proof}
    In case there exists $u \in W^{1,2}( \mathbb{D}, X )$ that satisfies $\mathrm{Area}^{\mu}(u ) > 0$, then the set where the approximate metric differential of $u$ is a norm has positive measure. Then a tangent $X_\omega$ of $X$ contains a two-dimensional normed space $V \subset X_\omega$ by a standard argument, cf. \cite[Proposition 11.2]{Lyt:Wen:17:areamini}. The isoperimetric set $\mathbb{I}_V \subset V$ of perimeter one satisfies
    \begin{align*}
        C_\mu \leq \mu_{V}( \mathbb{I}_V ).
    \end{align*}
    Let $Z$ be the injective hull of $V$ and $f \colon X_\omega \to Z$ $1$-Lipschitz such that $f|_{V} \colon V \to V$ is the identity. Let $c \colon \mathbb{S}^1 \to \partial \mathbb{I}_V$ be a constant speed parametrization. Since $\mu$ is a lower semi-continuous area, it follows that
    \begin{align*}
        \mu_{V}( \mathbb{I}_V ) = \mathrm{FillArea}_{V}^{\mu}( c ) = \mathrm{FillArea}_{Z}^{\mu}( f \circ c ),
    \end{align*}
    so the $1$-Lipschitz property of $f$ implies
    \begin{align*}
        C_\mu \leq \mathrm{FillArea}_{X_\omega}^{\mu}(c).
    \end{align*}
    By the stability result, \Cref{theorem-stability-dehn-function-coarse}, we deduce a contradiction
    \begin{align*}
        C_\mu
        \leq
        \mathrm{FillArea}_{X_\omega}( c )
        \leq
        \left( \limsup_{ r \rightarrow 0^{+} } \frac{ \delta^{\mu}_X(r) }{ r^2 } \right)
        <
        C_\mu.
    \end{align*}
\end{proof}

For the proof of \Cref{theorem-gromov's-theorem} and below, we will use the following definition: given a Jordan curve $\Gamma \subset X$, we define $\Lambda( \Gamma, X )$ to be the set of all $u \in W^{1,2}( \mathbb{D}, X )$ whose trace has a continuous representative which is a weakly monotone parametrization of $\Gamma$. 
\begin{proof}[Proof of \Cref{theorem-gromov's-theorem}]
The aim is to prove that every asymptotic cone $Y$ of $X$ is a metric tree. For this, it suffices to prove that there are no rectifiable Jordan curves $\Gamma \subset Y$. By the stability result, \Cref{theorem-stability-dehn-function-coarse}, every asymptotic cone $Y$ of $X$ satisfies
\begin{align*}
    \delta_{ Y }^{\mu}(r) \leq (1-\varepsilon) C_\mu r^2 \quad\text{for some $\varepsilon \in (0,1)$ and every $r > 0$.}
\end{align*}
Now suppose that $\Gamma \subset Y$ were a rectifiable Jordan curve of length $L > 0$. First $\mathrm{Area}(u) = 0$ for each $u \in \Lambda( \Gamma, Y )$ by \Cref{lemma-property-ET}. Then, using metric Morrey's $\varepsilon$-conformality, we find a sequence $u_m \in \Lambda( \Gamma, Y )$ satisfying a fixed 3-point condition and $E^{2}_{+}(u_m) \leq 1/m$. By Courant--Lebesgue, the sequence $( \trace( u_m ) )$ is equi-continuous, so, for the ultracompletion $Y_\omega$, the ultralimit $u_\omega \colon \Omega \to Y_\omega$ of $( u_m )$ belongs to $\Lambda( \Gamma, Y_\omega )$ and has zero energy, so it must be essentially constant. This is a contradiction because $u_\omega \in \Lambda( \Gamma, Y_\omega )$.
\end{proof}

\subsection{Comparison geometry}\label{section-comparison-geometry}
We prove some preliminary results before the proof of \Cref{theorem-local-cat(kappa)}.
\begin{lemma}\label{lemma-property-ET}
Suppose that $\mu$ is a lower semi-continuous area and that the inner product normed planes are the unique normed planes realizing the isoperimetric constant $C_\mu$. If $X$ is a complete length space with
\begin{align}\label{equation-asymptotic-dehn-function}
    \limsup_{ r \rightarrow 0^{+} } \frac{ \delta_{X}^{\mu}(r) }{ r^2 } \leq C_\mu,
\end{align}
then $X$ has property (ET). In particular, every $u \in W^{1,2}( \mathbb{D}, X )$ satisfies $$\mathrm{Area}^{\mu}(u) = \mathrm{Area}^{\mu_i}(u) = \mathrm{Area}(u).$$
\end{lemma}
Property (ET), namely that of having \emph{Euclidean tangents}, was introduced in \cite{Lyt:Wen:17:areamini}.
\begin{proof}
    In case $X$ does not have Property (ET), by definition, there exists $u \in W^{1,2}( \mathbb{D}, X )$ for which the set where the approximate metric differential of $u$ is not an inner product norm has positive measure. Then a tangent of $X$ contains a normed plane whose norm is not induced by an inner product. This leads to a contradiction as in the proof of \Cref{lemma-degenerate-area}.
%
\end{proof}

Before presenting the proof of \Cref{theorem-local-cat(kappa)}, we solve the Plateau problem in ultracompletions. To make this precise, we define further notation. The \emph{filling area} of a Jordan curve $\Gamma \subset X$ with respect to an area $\mu$ is
\begin{equation*}
    \mathrm{FillArea}^{\mu}_{X}( \Gamma )
    =
    \inf\left\{
        \mathrm{Area}^{\mu}(u) \colon u \in \Lambda( \Gamma, X )
    \right\}.
\end{equation*}
As in the case of Lipschitz curves, when $\mu$ is omitted, we refer to the parametrized Hausdorff area.

\begin{proposition}\label{proposition-existence-of-Plateau-problem}
Let $\kappa \in \R$ and $r_0 \in (0,2D_\kappa]$. Let $X$ be a complete length space and suppose that $\delta_{X}(r) \leq \delta_{\kappa}(r)$ for $r \in (0,r_0)$. If a rectifiable Jordan curve $\Gamma \subset X$ has length strictly less than $r_0$, then there exists a solution for the Plateau problem for $\Gamma$ in the ultracompletion $X_\omega$. That is, there exists $u_\omega \in \Lambda( \Gamma, X_\omega )$ satisfying
\begin{align*}
    E^{2}_{+}( u_\omega ) = \mathrm{Area}( u_\omega ) = \mathrm{FillArea}_{X_\omega}( \Gamma ).
\end{align*}
\end{proposition}
\begin{proof}
By the stability, \Cref{theorem-stability-dehn-function}, the ultracompletion $X_\omega$ satisfies
\begin{align}\label{equation-ultralimit-dehn-function-plateau}
    \delta_{X_\omega}(r) \leq \delta_\kappa(r) \quad\text{for $r \in (0,r_0)$.}
\end{align}
Consider a Jordan curve $\Gamma \subset X$ with length strictly less than $r_0$. Let $c \colon \mathbb{S}^1 \to \Gamma$ be a constant speed parametrization. Given \eqref{equation-ultralimit-dehn-function-plateau}, we obtain
\begin{align*}
    \mathrm{FillArea}_{X_\omega}( \Gamma ) = \mathrm{FillArea}_{X_\omega}( c )
\end{align*}
from \cite[Lemma 4.8]{Lyt:Wen:18:intrinsic}. Furthermore,
\begin{align*}
    \mathrm{Area}(v)
    =
    \mathrm{Area}^{\mu_i}(v)
    \quad\text{and}\quad
    \mathrm{FillArea}^{\mu_i}_{X_\omega}(c) = \mathrm{FillArea}_{X_\omega}(c)
\end{align*}
whenever $v \in W^{1,2}( \mathbb{D}, X_\omega )$ by \Cref{lemma-property-ET}. 
    
By \cite[Proposition 4.1]{stadler-wenger-2025-isoperimetric-inequalities-vs-upper-curvature-bounds} (cf. the proof of \cite[Proposition 3.5]{Lyt:Wen:You:20}), there exists a sequence of complete length spaces $( Y_m )$ such that $X \hookrightarrow Y_m$ isometrically and the following holds: First, the $m^{-1}$-neighbourhood of $X$ in $Y_m$ is $Y_m$. Second, $Y_m$ is $L$-Lipschitz $1$-connected up to scale $\frac{ 1 }{ L m }$. Third,
\begin{align*}
    \delta^{ \mu_i }_{ Y_m }(r) \leq \delta_{\kappa}(r) + L r^2 \quad\text{for $r \in (0, r_0)$}
\end{align*}
and that
\begin{align*}
    \delta^{ \mu_i }_{ Y_m }(r) \leq \delta_{\kappa}(r) + \frac{ 1 }{ \sqrt{m} } r^2 \quad\text{for $r \in \left( \frac{1}{\sqrt{m}}, r_0 \right)$,}
\end{align*}
where $L$ is a universal constant.

By metric Morrey's $\varepsilon$-conformality, there exists $u_m \in \Lambda( \Gamma, Y_m )$ so that
\begin{align*}
    E^{2}_{+}( u_m ) \leq \mathrm{FillArea}^{\mu_i}_{Y_m}( \Gamma ) + \frac{ 1 }{ m }
\end{align*}
and $\trace( u_m )$ satisfies a fixed $3$-point condition. By Courant--Lebesgue, the family $( \trace( u_m ) )$ is equi-continuous.

Let $X_\omega$ be an ultralimit $( Y_m, d_m, p )$ for $p \in X \subset Y_m$. Notice that $X_\omega$ is isometric to the ultracompletion of $X$. The ultralimit $u_\omega \in W^{1,2}( \mathbb{D}, X_\omega )$ of $( u_m )$ has the following key properties. First, the trace $\trace( u_\omega )$ is an ultralimit of $( \trace( u_m ) )$ and thus $u_\omega \in \Lambda( \Gamma, X_\omega )$. Second,
\begin{align*}
    \mathrm{Area}(u_\omega) = \mathrm{Area}^{\mu_i}( u_\omega ) \leq E^{2}_{+}( u_\omega ) \leq \lim\nolimits_{\omega} E^{2}_{+}( u_m )
\end{align*}
and
\begin{align*}
    \lim\nolimits_{\omega} E^{2}_{+}( u_m )
    \leq
    \lim\nolimits_{\omega} \mathrm{FillArea}^{\mu_i}_{Y_m}( \Gamma ).
\end{align*}
Third, the lifting result, \cite[Theorem 5.1]{Wen:2019}, extends to
\begin{align*}
    \lim\nolimits_{\omega} \mathrm{FillArea}^{\mu_i}_{Y_m}( c ) \leq \mathrm{FillArea}^{\mu_i}_{X_\omega}(c)
\end{align*}
with minor modifications. Finally, since $\mathrm{FillArea}^{\mu_i}_{Y_m}( \Gamma ) = \mathrm{FillArea}^{\mu_i}_{Y_m}( c )$ by \cite[Lemma 4.8]{Lyt:Wen:18:intrinsic}, we have
\begin{align*}
    \lim\nolimits_{\omega} \mathrm{FillArea}^{\mu_i}_{Y_m}( \Gamma ) \leq \mathrm{FillArea}_{ X_\omega }( \Gamma ) \leq \mathrm{Area}( u_\omega ).
\end{align*}
The claim follows from the chains of inequalities.
\end{proof}

Before the proof of \Cref{theorem-local-cat(kappa)}, we recall the definition of upper curvature bounds on metric spaces. A \emph{geodesic triangle} $\Delta \subset X$ is a union of three geodesics $\gamma_1$, $\gamma_2$, and $\gamma_3$, joining the \emph{vertices} $p_1$ to $p_2$, $p_2$ to $p_3$, and $p_3$ to $p_1$, respectively. The \emph{perimeter} of $\Delta$ is $\ell( \gamma_1 ) + \ell( \gamma_2 ) + \ell( \gamma_3 )$.

Given $\kappa \in \R$, a geodesic triangle $\Delta$ with perimeter $< 2D_\kappa$ satisfies the $\mathrm{CAT}(\kappa)$ comparison if, given a comparison triangle $\Delta' \subset M^{2}_\kappa$, the \emph{comparison map} $\Delta' \to \Delta$, sending sides to sides isometrically, is $1$-Lipschitz. We say that a metric space $X$ is $\mathrm{CAT}(\kappa)$ if every geodesic triangle with perimeter $< 2D_\kappa$ satisfies the $\mathrm{CAT}(\kappa)$ comparison and every $x, y \in X$ with $d(x,y) < D_\kappa$ can be joined with a geodesic; see \cite{bridson-haefliger-1999-metric-spaces-of-non-positive-curvature} (and \cite{alexander-kapovitch-petrunin-2024-alexandrov-geometry-foundations}) for further background. A \emph{Jordan triangle} is a geodesic triangle so that the sides overlap only at their end points. In particular, a Jordan triangle is homeomorphic to $\S^1$.

\begin{proof}[Proof of \Cref{theorem-local-cat(kappa)}]
    We first prove the following claim: Let $X$ be a complete length space with $\delta_{X}(r) \leq \delta_{\kappa}(r)$ for $r \in (0,r_0)$. Then any geodesic triangle $\Delta \subset X$ with perimeter $< r_0$ satisfies the $\mathrm{CAT}(\kappa)$ comparison.

    Arguing as in \cite[Lemma 3.1]{Lyt:Wen:18:CAT}, it suffices to consider a Jordan triangle $\Delta$ with perimeter $< r_0$. Consider the ultracompletion $Y = X_\omega$. By the stability, \Cref{theorem-stability-dehn-function}, the ultracompletion satisfies
    \begin{align}\label{equation-ultralimit-dehn-function}
	   \delta_{Y}(r) \leq \delta_\kappa(r) \quad\text{for $r \in (0,r_0)$.}
    \end{align}
    Let $u \in \Lambda(\Delta, Y)$ be a solution to the Plateau problem obtained from \Cref{proposition-existence-of-Plateau-problem}. It holds by \eqref{equation-ultralimit-dehn-function} that $u$ has a continuous representative up to the boundary, cf. \cite[Theorem 9.1]{Lyt:Wen:17:areamini}. Moreover, $u$ factors through an intrinsic disk $Z$: by \cite[Theorem 1.1]{Lyt:Wen:18:intrinsic}, there exists a geodesic space $Z$ homeomorphic to $\overline{\mathbb{D}}$, a continuous and monotone $P_u \colon \overline{\mathbb{D}} \to Z$, and $1$-Lipschitz $\overline{u} \colon Z \to X_\omega$ satisfying $u = \overline{u} \circ P_u$, and that the lengths of $u \circ \gamma$ and $P_u \circ \gamma$ are equal for every rectifiable $\gamma$. Hence it suffices to verify the comparison estimate for $\partial Z$. We claim that
\begin{align}\label{equation-intrinsic-CAT-kappa-property}
    \mathcal{H}^{2}( \Omega ) \leq \delta_{\kappa}( \ell( \partial \Omega ) )
    \quad\text{for Jordan domains $\Omega \subset Z$ with $\ell( \partial \Omega ) < 2D_\kappa$.}
\end{align}
    To this end, when $\ell( \partial \Omega ) < r_0$, inequality \eqref{equation-intrinsic-CAT-kappa-property} follows from \eqref{equation-ultralimit-dehn-function} by the proof of \cite[Theorem 8.2]{Lyt:Wen:18:intrinsic} with minor modifications. If $r_0 \leq \ell( \partial \Omega ) < 2D_\kappa$ instead, we observe $\mathcal{H}^{2}(\Omega) \leq \mathcal{H}^{2}(Z) \leq \delta_\kappa( \ell(\partial Z) ) \leq \delta_{\kappa}( \ell(\partial \Omega) )$. Then, by the proof of \cite[Theorem 1.4]{Lyt:Wen:20}, we conclude    $\delta_{Z}(r) \leq \delta_{\kappa}(r)$ for $r \in (0,2D_\kappa)$. Thus, by \cite[Theorem 1.4]{Lyt:Wen:18:CAT}, $Z$ is $\mathrm{CAT}(\kappa)$. The proof of the claim is complete.

    To finish, we first assume that $X$ is geodesic. If $r_0 = 2 D_\kappa$, $X$ is $\mathrm{CAT}(\kappa)$ by the claim on geodesic triangles. If $r_0 < 2D_\kappa$, we proceed as follows. For all $x, y \in X$ with $d(x,y) < r_0/2$, the claim on geodesic triangles implies that the geodesic $\gamma_{xy}$ joining $x$ to $y$ is unique and moreover depends continuously on the end points. Furthermore, if $x, y$ are contained in an open ball $B = B_{X}( x_0, r_0/4 )$, then $\gamma_{xy}$ lies in $B$. In particular, $B$ is geodesic and locally $\mathrm{CAT}(\kappa)$. By \cite[Radial Lemma, Lemma 9.52]{alexander-kapovitch-petrunin-2024-alexandrov-geometry-foundations}, $B$ is $\mathrm{CAT}(\kappa)$ and it is not difficult to see that the closed ball $\overline{B}$, which coincides with the closure of $B$, is geodesic and $\mathrm{CAT}(\kappa)$, as claimed.

    It remains to consider the case when $X$ is merely a complete length space. Then its ultracompletion $Y$ is geodesic and satisfies the conclusions of the previous steps. Since $X$ is a complete length space, it easily follows that for any $x,y \in X$ with $d(x,y) < r_0/2$, the unique geodesic $\gamma_{xy}$ in $Y$ is, in fact, contained in $X$. In case $r_0 = 2D_\kappa$, the proof is complete. If $r_0 < 2D_\kappa$, then any closed ball $\overline{B}_{X}( x_0, r_0/4 )$ is a geodesic subset of the $\mathrm{CAT}(\kappa)$ space $\overline{B}_{Y}( x_0, r_0/4 )$. Thus $\overline{B}_{X}( x_0, r_0/4 )$ is $\mathrm{CAT}(\kappa)$.
\end{proof}

\begin{remark}
\Cref{proposition-existence-of-Plateau-problem}, and thus \Cref{theorem-local-cat(kappa)}, clearly holds if the Busemann--Hausdorff area is replaced by any lower semi-continuous area $\mu$ for which the inner product normed planes are the unique normed planes realizing the isoperimetric constant $C_\mu = 1/(4\pi)$. This is immediate from \Cref{lemma-property-ET}.
\end{remark}

\bibliographystyle{alpha}

\begin{thebibliography}{GHWX22}

\bibitem[AKP24]{alexander-kapovitch-petrunin-2024-alexandrov-geometry-foundations}
Stephanie Alexander, Vitali Kapovitch, and Anton Petrunin.
\newblock {\em Alexandrov geometry---foundations}, volume 236 of {\em Graduate Studies in Mathematics}.
\newblock American Mathematical Society, Providence, RI, [2024] \copyright 2024.

\bibitem[APT04]{alvarez-thompson-2004-volumes-on-normed-and-finsler-spaces}
J.~C. \'Alvarez~Paiva and A.~C. Thompson.
\newblock Volumes on normed and {F}insler spaces.
\newblock In {\em A sampler of {R}iemann-{F}insler geometry}, volume~50 of {\em Math. Sci. Res. Inst. Publ.}, pages 1--48. Cambridge Univ. Press, Cambridge, 2004.

\bibitem[Bat22]{bate-2022-characterising-rectifiable-metric-spaces-using-tangent-spaces}
David Bate.
\newblock Characterising rectifiable metric spaces using tangent spaces.
\newblock {\em Invent. Math.}, 230(3):995--1070, 2022.

\bibitem[BH99]{bridson-haefliger-1999-metric-spaces-of-non-positive-curvature}
Martin~R. Bridson and Andr\'e Haefliger.
\newblock {\em Metric spaces of non-positive curvature}, volume 319 of {\em Grundlehren der mathematischen Wissenschaften [Fundamental Principles of Mathematical Sciences]}.
\newblock Springer-Verlag, Berlin, 1999.

\bibitem[BI12]{burago-ivanov-2012-minimality-of-planes-in-normed-spaces}
Dmitri Burago and Sergei Ivanov.
\newblock Minimality of planes in normed spaces.
\newblock {\em Geom. Funct. Anal.}, 22(3):627--638, 2012.

\bibitem[BNS22]{brue-naber-semola-2022-boundary-regularity-and-stability-for-spaces-with-ricci-bounded-below}
Elia Bru\`e, Aaron Naber, and Daniele Semola.
\newblock Boundary regularity and stability for spaces with {R}icci bounded below.
\newblock {\em Invent. Math.}, 228(2):777--891, 2022.

\bibitem[BPS25]{brue-pigati-semola-2025-topological-regularity-and-stability-of-noncollapsed-spaces-with-ricci-curvature-bounded-below}
Elia Bruè, Alessandro Pigati, and Daniele Semola.
\newblock Topological regularity and stability of noncollapsed spaces with ricci curvature bounded below, 2025.
\newblock cvgmt preprint.

\bibitem[CC97]{cheeger-colding-1997-on-the-structure-of-spaces-with-Ricci-curvature-bounded-below-I}
Jeff Cheeger and Tobias~H. Colding.
\newblock On the structure of spaces with {R}icci curvature bounded below. {I}.
\newblock {\em J. Differential Geom.}, 46(3):406--480, 1997.

\bibitem[CC00a]{cheeger-colding-1997-on-the-structure-of-spaces-with-Ricci-curvature-bounded-below-II}
Jeff Cheeger and Tobias~H. Colding.
\newblock On the structure of spaces with {R}icci curvature bounded below. {II}.
\newblock {\em J. Differential Geom.}, 54(1):13--35, 2000.

\bibitem[CC00b]{cheeger-colding-1997-on-the-structure-of-spaces-with-Ricci-curvature-bounded-below-III}
Jeff Cheeger and Tobias~H. Colding.
\newblock On the structure of spaces with {R}icci curvature bounded below. {III}.
\newblock {\em J. Differential Geom.}, 54(1):37--74, 2000.

\bibitem[CE24]{Creu:Evs:24}
Paul Creutz and Nikita Evseev.
\newblock An approach to metric space-valued {S}obolev maps via weak{$*$} derivatives.
\newblock {\em Anal. Geom. Metr. Spaces}, 12(1):Paper No. 20230107, 16, 2024.

\bibitem[Cre22]{Creu:22}
Paul Creutz.
\newblock Plateau's problem for singular curves.
\newblock {\em Comm. Anal. Geom.}, 30(8):1779--1792, 2022.

\bibitem[DM21]{daskalopoulos-mese-2021-rigidity-of-teichmuller-space}
Georgios Daskalopoulos and Chikako Mese.
\newblock Rigidity of {T}eichm\"uller space.
\newblock {\em Invent. Math.}, 224(3):791--916, 2021.

\bibitem[DMV11]{daskalopoulos-mese-2011-superrigidity-of-hyperbolic-buildings}
Georgios Daskalopoulos, Chikako Mese, and Alina Vdovina.
\newblock Superrigidity of hyperbolic buildings.
\newblock {\em Geom. Funct. Anal.}, 21(4):905--919, 2011.

\bibitem[Dru02]{drutu-2002-quasi-isometry-invariants-and-asymptotic-cones}
Cornelia Dru{\c t}u.
\newblock Quasi-isometry invariants and asymptotic cones.
\newblock volume~12, pages 99--135. 2002.
\newblock International Conference on Geometric and Combinatorial Methods in Group Theory and Semigroup Theory (Lincoln, NE, 2000).

\bibitem[DS05]{drutu-sapir-2005-tree-graded-spaces-and-asymptotic-cones-of-groups}
Cornelia Dru{\c t}u and Mark Sapir.
\newblock Tree-graded spaces and asymptotic cones of groups.
\newblock {\em Topology}, 44(5):959--1058, 2005.
\newblock With an appendix by Denis Osin and Mark Sapir.

\bibitem[Fed69]{Fed:69}
Herbert Federer.
\newblock {\em Geometric measure theory}.
\newblock Die Grundlehren der mathematischen Wissenschaften, Band 153. Springer-Verlag New York Inc., New York, 1969.

\bibitem[FW20]{Fitz:Wen:20}
Martin Fitzi and Stefan Wenger.
\newblock Morrey's {$\varepsilon$}-conformality lemma in metric spaces.
\newblock {\em Proc. Amer. Math. Soc.}, 148(10):4285--4298, 2020.

\bibitem[GHR03]{gersten-holt-riley-2003-isoperimetric-inequalities-for-nilpotent-groups}
S.~M. Gersten, D.~F. Holt, and T.~R. Riley.
\newblock Isoperimetric inequalities for nilpotent groups.
\newblock {\em Geom. Funct. Anal.}, 13(4):795--814, 2003.

\bibitem[GHWX22]{guo-huang-wang-xu-2022-p-harmonic-mappings-between-metric-spaces}
Chang-Yu Guo, Manzi Huang, Zhuang Wang, and Haiqing Xu.
\newblock {$p$}-harmonic mappings between metric spaces.
\newblock {\em Math. Z.}, 302(3):1797--1819, 2022.

\bibitem[GPW90]{grove-petersen-wu-1990-geometric-finiteness-theorems-via-controlled-topology}
Karsten Grove, Peter Petersen, V, and Jyh~Yang Wu.
\newblock Geometric finiteness theorems via controlled topology.
\newblock {\em Invent. Math.}, 99(1):205--213, 1990.

\bibitem[Gro81]{gromov-1981-groups-of-polynomial-growth-and-expanding-maps}
Mikhael Gromov.
\newblock Groups of polynomial growth and expanding maps.
\newblock {\em Inst. Hautes \'Etudes Sci. Publ. Math.}, (53):53--73, 1981.

\bibitem[Gro87]{gromov-1987-hyperbolic-groups}
M.~Gromov.
\newblock Hyperbolic groups.
\newblock In {\em Essays in group theory}, volume~8 of {\em Math. Sci. Res. Inst. Publ.}, pages 75--263. Springer, New York, 1987.

\bibitem[Gro93]{gromov-1993-asymptotic-invariants-of-infinite-groups}
M.~Gromov.
\newblock Asymptotic invariants of infinite groups.
\newblock In {\em Geometric group theory, {V}ol.\ 2 ({S}ussex, 1991)}, volume 182 of {\em London Math. Soc. Lecture Note Ser.}, pages 1--295. Cambridge Univ. Press, Cambridge, 1993.

\bibitem[Gro07]{Gro:07}
Misha Gromov.
\newblock {\em Metric structures for {R}iemannian and non-{R}iemannian spaces}.
\newblock Modern Birkh\"{a}user Classics. Birkh\"{a}user Boston, Inc., Boston, MA, english edition, 2007.
\newblock Based on the 1981 French original, With appendices by M. Katz, P. Pansu and S. Semmes, Translated from the French by Sean Michael Bates.

\bibitem[GS92]{gromov-schoen-1992-harmonic-maps-into-singular-spaces-and-p-adic-superrigidity-for-lattices-in-groups-of-rank-one}
Mikhail Gromov and Richard Schoen.
\newblock Harmonic maps into singular spaces and {$p$}-adic superrigidity for lattices in groups of rank one.
\newblock {\em Inst. Hautes \'Etudes Sci. Publ. Math.}, (76):165--246, 1992.

\bibitem[GW20]{guo-wenger-2020-area-minimizing-discs-in-locally-non-compact-metric-spaces}
Chang-Yu Guo and Stefan Wenger.
\newblock Area minimizing discs in locally non-compact metric spaces.
\newblock {\em Comm. Anal. Geom.}, 28(1):89--112, 2020.

\bibitem[Haj96]{Haj:96}
Piotr Haj{\l}asz.
\newblock Sobolev spaces on an arbitrary metric space.
\newblock {\em Potential Anal.}, 5(4):403--415, 1996.

\bibitem[HKST01]{Hei:Kos:Sha:Ty:01}
Juha Heinonen, Pekka Koskela, Nageswari Shanmugalingam, and Jeremy~T. Tyson.
\newblock Sobolev classes of {B}anach space-valued functions and quasiconformal mappings.
\newblock {\em J. Anal. Math.}, 85:87--139, 2001.

\bibitem[HKST15]{Hei:Kos:Sha:Ty:15}
Juha Heinonen, Pekka Koskela, Nageswari Shanmugalingam, and Jeremy~T. Tyson.
\newblock {\em Sobolev spaces on metric measure spaces}, volume~27 of {\em New Mathematical Monographs}.
\newblock Cambridge University Press, Cambridge, 2015.
\newblock An approach based on upper gradients.

\bibitem[HS14]{hajlasz-schikorra-2014-lipschitz-homotopy-and-density-of-lipschitz-mappings-in-sobolev-spaces}
Piotr Haj{\l}asz and Armin Schikorra.
\newblock Lipschitz homotopy and density of {L}ipschitz mappings in {S}obolev spaces.
\newblock {\em Ann. Acad. Sci. Fenn. Math.}, 39(2):593--604, 2014.

\bibitem[Isb64]{isbell-1964-six-theorems-about-injective-metric-spaces}
J.~R. Isbell.
\newblock Six theorems about injective metric spaces.
\newblock {\em Comment. Math. Helv.}, 39:65--76, 1964.

\bibitem[Iva08]{ivanov-2008-volumes-and-areas-lipschitz-metrics}
S.~V. Ivanov.
\newblock Volumes and areas of {L}ipschitz metrics.
\newblock {\em Algebra i Analiz}, 20(3):74--111, 2008.

\bibitem[Jos94]{jost-1994-equilibrium-maps-between-metric-spaces}
J\"urgen Jost.
\newblock Equilibrium maps between metric spaces.
\newblock {\em Calc. Var. Partial Differential Equations}, 2(2):173--204, 1994.

\bibitem[Kar07]{Karm:07}
M.~B. Karmanova.
\newblock Area and co-area formulas for mappings of the {S}obolev classes with values in a metric space.
\newblock {\em Sibirsk. Mat. Zh.}, 48(4):778--788, 2007.

\bibitem[KKL98]{kapovich-kleiner-1998-quasi-isometries-and-the-de-rham-decomposition}
Michael Kapovich, Bruce Kleiner, and Bernhard Leeb.
\newblock Quasi-isometries and the de {R}ham decomposition.
\newblock {\em Topology}, 37(6):1193--1211, 1998.

\bibitem[KL95]{kapovich-leeb-1995-on-asymptotic-cones-and-quasi-isometry-classes-of-fundamental-groups-of-3-manifolds}
M.~Kapovich and B.~Leeb.
\newblock On asymptotic cones and quasi-isometry classes of fundamental groups of {$3$}-manifolds.
\newblock {\em Geom. Funct. Anal.}, 5(3):582--603, 1995.

\bibitem[KL97]{kleiner-leeb-1997-rigidity-of-quasi-isometries-for-symmetric-spaces-and-euclidean-buildings}
Bruce Kleiner and Bernhard Leeb.
\newblock Rigidity of quasi-isometries for symmetric spaces and {E}uclidean buildings.
\newblock {\em Inst. Hautes \'Etudes Sci. Publ. Math.}, (86):115--197, 1997.

\bibitem[KS93]{Kor:Sch:93}
Nicholas~J. Korevaar and Richard~M. Schoen.
\newblock Sobolev spaces and harmonic maps for metric space targets.
\newblock {\em Comm. Anal. Geom.}, 1(3-4):561--659, 1993.

\bibitem[KSTT05]{kramer-shelah-tent-thomas-2005-asymptotic-cones-of-finitely-presented-groups}
Linus Kramer, Saharon Shelah, Katrin Tent, and Simon Thomas.
\newblock Asymptotic cones of finitely presented groups.
\newblock {\em Adv. Math.}, 193(1):142--173, 2005.

\bibitem[LS05]{lang-schlichenmaier-2005-nagata-dimension-quasisymmetric-embeddings-and-lipschitz-extensions}
Urs Lang and Thilo Schlichenmaier.
\newblock Nagata dimension, quasisymmetric embeddings, and {L}ipschitz extensions.
\newblock {\em Int. Math. Res. Not.}, (58):3625--3655, 2005.

\bibitem[LV09]{lott-villani-2009-ricci-curvature-for-metric-measure-spaces-via-optimal-transport}
John Lott and C\'edric Villani.
\newblock Ricci curvature for metric-measure spaces via optimal transport.
\newblock {\em Ann. of Math. (2)}, 169(3):903--991, 2009.

\bibitem[LW16]{Lyt:Wen:16}
Alexander Lytchak and Stefan Wenger.
\newblock Regularity of harmonic discs in spaces with quadratic isoperimetric inequality.
\newblock {\em Calc. Var. Partial Differential Equations}, 55(4):Art. 98, 19, 2016.

\bibitem[LW17a]{Lyt:Wen:17:areamini}
Alexander Lytchak and Stefan Wenger.
\newblock Area minimizing discs in metric spaces.
\newblock {\em Arch. Ration. Mech. Anal.}, 223(3):1123--1182, 2017.

\bibitem[LW17b]{Lyt:Wen:17:energyarea}
Alexander Lytchak and Stefan Wenger.
\newblock Energy and area minimizers in metric spaces.
\newblock {\em Adv. Calc. Var.}, 10(4):407--421, 2017.

\bibitem[LW18a]{Lyt:Wen:18:intrinsic}
Alexander Lytchak and Stefan Wenger.
\newblock Intrinsic structure of minimal discs in metric spaces.
\newblock {\em Geom. Topol.}, 22(1):591--644, 2018.

\bibitem[LW18b]{Lyt:Wen:18:CAT}
Alexander Lytchak and Stefan Wenger.
\newblock Isoperimetric characterization of upper curvature bounds.
\newblock {\em Acta Math.}, 221(1):159--202, 2018.

\bibitem[LW20]{Lyt:Wen:20}
Alexander Lytchak and Stefan Wenger.
\newblock Canonical parameterizations of metric disks.
\newblock {\em Duke Math. J.}, 169(4):761--797, 2020.

\bibitem[LWY20]{Lyt:Wen:You:20}
Alexander Lytchak, Stefan Wenger, and Robert Young.
\newblock Dehn functions and {H}\"older extensions in asymptotic cones.
\newblock {\em J. Reine Angew. Math.}, 763:79--109, 2020.

\bibitem[MW25]{meier-wenger-2025-quasiconformal-almost-parametrizations-of-metric-surfaces}
Damaris Meier and Stefan Wenger.
\newblock Quasiconformal almost parametrizations of metric surfaces.
\newblock {\em J. Eur. Math. Soc. (JEMS)}, 27(12):5133--5154, 2025.

\bibitem[NR22]{Nta:Rom:22}
Dimitrios {Ntalampekos} and Matthew {Romney}.
\newblock {Polyhedral approximation and uniformization for non-length surfaces}.
\newblock {\em arXiv e-prints}, page arXiv:2206.01128, June 2022.
\newblock to appear in J. Eur. Math. Soc. (JEMS).

\bibitem[NR23]{ntalampekos-romney-2023-polyhedral-approximation-of-metric-surfaces-and-applications-to-uniformization}
Dimitrios Ntalampekos and Matthew Romney.
\newblock Polyhedral approximation of metric surfaces and applications to uniformization.
\newblock {\em Duke Math. J.}, 172(9):1673--1734, 2023.

\bibitem[Pan83]{pansu-1983-croissance-des-boules-et-geodesiques-fermees-dans-les-nilvarietes}
Pierre Pansu.
\newblock Croissance des boules et des g\'eod\'esiques ferm\'ees dans les nilvari\'et\'es.
\newblock {\em Ergodic Theory Dynam. Systems}, 3(3):415--445, 1983.

\bibitem[Pet90]{petersen-1990-a-finiteness-theorem-for-metric-spaces}
Peter Petersen, V.
\newblock A finiteness theorem for metric spaces.
\newblock {\em J. Differential Geom.}, 31(2):387--395, 1990.

\bibitem[Raj17]{Raj:17}
Kai Rajala.
\newblock Uniformization of two-dimensional metric surfaces.
\newblock {\em Invent. Math.}, 207(3):1301--1375, 2017.

\bibitem[Res68]{reshetnyak-1968-non-expansive-maps-in-a-space-of-curvature-no-greater-than-K}
Yu.~G. Reshetnyak.
\newblock Non-expansive maps in a space of curvature no greater than {$K$}.
\newblock {\em Sibirsk. Mat. Zh.}, 9:918--927, 1968.

\bibitem[Res97]{Res:97}
Yu.~G. Reshetnyak.
\newblock Sobolev classes of functions with values in a metric space.
\newblock {\em Sibirsk. Mat. Zh.}, 38(3):657--675, iii--iv, 1997.

\bibitem[Sha00]{Sha:00}
Nageswari Shanmugalingam.
\newblock Newtonian spaces: an extension of {S}obolev spaces to metric measure spaces.
\newblock {\em Rev. Mat. Iberoamericana}, 16(2):243--279, 2000.

\bibitem[Sta21]{stadler-the-structure-of-minimal-surfaces-in-cat-0-spaces}
Stephan Stadler.
\newblock The structure of minimal surfaces in {CAT}(0) spaces.
\newblock {\em J. Eur. Math. Soc. (JEMS)}, 23(11):3521--3554, 2021.

\bibitem[Stu06a]{sturm-2006-on-the-geometry-of-metric-measure-spaces-I}
Karl-Theodor Sturm.
\newblock On the geometry of metric measure spaces. {I}.
\newblock {\em Acta Math.}, 196(1):65--131, 2006.

\bibitem[Stu06b]{sturm-2006-on-the-geometry-of-metric-measure-spaces-II}
Karl-Theodor Sturm.
\newblock On the geometry of metric measure spaces. {II}.
\newblock {\em Acta Math.}, 196(1):133--177, 2006.

\bibitem[SW01]{Schoen:Wolfson:01}
R.~Schoen and J.~Wolfson.
\newblock Minimizing area among {L}agrangian surfaces: the mapping problem.
\newblock {\em J. Differential Geom.}, 58(1):1--86, 2001.

\bibitem[SW25]{stadler-wenger-2025-isoperimetric-inequalities-vs-upper-curvature-bounds}
Stephan Stadler and Stefan Wenger.
\newblock Isoperimetric inequalities vs upper curvature bounds.
\newblock {\em Geom. Topol.}, 29(2):829--862, 2025.

\bibitem[vdDW84]{van-der-dries-wilkie-1984-gromovs-theorem-on-groups-of-polynomial-growth-and-elementary-logic}
L.~van~den Dries and A.~J. Wilkie.
\newblock Gromov's theorem on groups of polynomial growth and elementary logic.
\newblock {\em J. Algebra}, 89(2):349--374, 1984.

\bibitem[Vod00]{Vod:00}
Serguei~K. Vodop'yanov.
\newblock {{\(\mathcal P\)}}-differentiability on {Carnot} groups in different topologies and related topics.
\newblock In {\em Trudy po analizu i geometrii}, pages 603--670. Novosibirsk: Izdatel'stvo Instituta Matematiki Im. S. L. Soboleva SO RAN, 2000.

\bibitem[Wen08]{wenger-2008-gromov-hyperbolic-spaces-and-the-sharp-isoperimetric-constant}
Stefan Wenger.
\newblock Gromov hyperbolic spaces and the sharp isoperimetric constant.
\newblock {\em Invent. Math.}, 171(1):227--255, 2008.

\bibitem[Wen11a]{wenger-2011-compactness-for-manifolds-and-integral-currents-with-bounded-diameter-and-volume}
Stefan Wenger.
\newblock Compactness for manifolds and integral currents with bounded diameter and volume.
\newblock {\em Calc. Var. Partial Differential Equations}, 40(3-4):423--448, 2011.

\bibitem[Wen11b]{wenger-2011-nilpotent-groups-without-exactly-polynomial-dehn-function}
Stefan Wenger.
\newblock Nilpotent groups without exactly polynomial {D}ehn function.
\newblock {\em J. Topol.}, 4(1):141--160, 2011.

\bibitem[Wen19]{Wen:2019}
Stefan Wenger.
\newblock Spaces with almost {E}uclidean {D}ehn function.
\newblock {\em Math. Ann.}, 373(3-4):1177--1210, 2019.

\bibitem[You13]{young-2013-the-dehn-function-of-slnz}
Robert Young.
\newblock The {D}ehn function of {${\rm SL}(n;\Bbb Z)$}.
\newblock {\em Ann. of Math. (2)}, 177(3):969--1027, 2013.

\end{thebibliography}

\end{document}